\documentclass[a4paper,twoside,11pt]{article}

\usepackage{color}

\usepackage[english]{babel}
\usepackage[latin1]{inputenc}
\usepackage{lmodern}
\usepackage[T1]{fontenc}
\usepackage[a4paper,top=3cm,bottom=3cm,left=3cm,right=3cm,
bindingoffset=5mm]{geometry}
\usepackage{authblk}
\usepackage{emp}
\usepackage{amsmath,amssymb,amsthm}
\usepackage{graphicx}
\usepackage{multirow}

\newtheorem{theorem}{Theorem}[section]

\newtheorem{corollary}{Corollary}[section]
\newtheorem{lemma}{Lemma}[section]
\newtheorem{remark}{Remark}[section]

\author{
L. Beir\~ao da Veiga
\thanks{Dipartimento di Matematica,  Universit\`a degli Studi di Milano, Italy, E-mail: {\tt lourenco.beirao@unimi.it}},
C. Lovadina
\thanks{Dipartimento di Matematica, Universit\`a di Pavia, Italy, E-mail: {\tt carlo.lovadina@unipv.it}},
D. Mora
\thanks{ GIMNAP, Departamento de Matem\'atica, Universidad del B\'io-B\'io, Casilla
5-C, Concepci\'on, Chile, and Centro de Investigaci\'on en Ingenier\'ia
Matem\'atica (CI2MA), Universidad de Concepci\'on, Concepci\'on, Chile, E-mail: {\tt dmora@ubiobio.cl}
}
}

\date{}

\title{A Virtual Element Method for elastic and inelastic problems on polytope meshes}

\def\f{{\bf f}}
\def\u{{\bf u}}
\def\v{{\bf v}}
\def\w{{\bf w}}
\def\s{{\bf s}}
\def\n{{\bf n}}
\def\t{{\bf t}}
\def\r{{\bf r}}
\def\q{{\bf q}}

\def\sigmab{\boldsymbol{\sigma}}
\def\deltab{\boldsymbol{\delta}}
\def\taub{\boldsymbol{\tau}}

\def\Pb{\boldsymbol{P}}
\def\Id{\boldsymbol{I}}

\def\V{\mathcal{V}}
\def\W{\mathcal{W}}
\def\Pp{{\cal P}}
\def\Ppb{\boldsymbol{\cal P}}
\def\VhE{V_{h,E}}
\def\Vh{V_{h}}
\def\Vhf{V_{h,f}}

\def\div{\textrm{div}}
\def\dx{\:\textrm{d}x}
\def\ds{\:\textrm{d}s}
\def\ah{a_h}
\def\ahE{a_{h,E}}

\def\athE{\widetilde{a}_{h,E}}

\def\dev{{\rm dev}}
\def\tr{{\rm tr}}

\def\Th{\Omega_h}
\def\Eh{{\cal E}_h}
\def\vb{{\bf v}}
\def\ub{{\bf u}}
\def\tb{{\bf t}}
\def\calV{\cal V}
\def\vv{{\mathsf v}}

\begin{document}

\maketitle

\begin{abstract}
We present a Virtual Element Method (VEM) for possibly nonlinear
elastic and inelastic problems, mainly focusing on a small
deformation regime. The numerical scheme is based on a low-order
approximation of the displacement field, as well as a suitable
treatment of the displacement gradient. The proposed method allows
for general polygonal and polyhedral meshes, it is efficient in terms
of number of applications of the constitutive law, and it can make
use of any standard black-box constitutive law algorithm. 
Some theoretical results have been developed for the elastic case.
Several numerical results within the 2D setting are presented, and a brief discussion on
the extension to large deformation problems is included.
\end{abstract}

\section{Introduction}
\label{sec:intro}

The Virtual Element Method (VEM), introduced in \cite{volley}, is a recent generalization
 of the Finite Element Method which is characterized by the capability of dealing with very general
polygonal/polyhedral meshes and the possibility to easily implement
highly regular discrete spaces \cite{kirchhoff,arbitrary}. Indeed,
by avoiding the explicit construction of the local basis
functions, the VEM can easily handle general polygons/polyhedrons
without complex integrations on the element (see
\cite{hitchhikers} for details on the coding aspects of the
method). 
The interest in numerical methods that can make use of general
polytopal meshes has recently undergone a significant growth in the
mathematical and engineering literature. Among the large number of
papers, we cite as a minimal sample
\cite{Brezzi-Lipnikov-Shashkov:2005,MFD-book,Sukumar-Tabarraei:2004,Gillette-2,TPPM10,Wang-1,Wang-2,DiPietro-Ern-1,DiPietro-Ern-3,Cockburn-Jay-Lazarov}.
Indeed, polytopal meshes can be very useful for a wide range of
reasons, including meshing of the domain (such as cracks) and data
(such as inclusions) features, automatic use of hanging nodes, use
of moving meshes, adaptivity.

In the framework of Structural Mechanics,
recent applications of Polygonal Finite Element Methods, which is a different
technology employing direct integration of complex non-polynomial
functions, have shed light on some very interesting advantages of using general
polygons to mesh the computational domain. This include, for instance, the greater robustness to
mesh distortion \cite{Paulino-nonlinear-polygonal}, a reduced mesh sensitivity of
solutions in topology optimization \cite{TPPM10,Topology-VEM},
better handling of contact problems \cite{Wriggers-2014} and crack
propagation \cite{Paulino-cracks}. Unfortunately, Polygonal Finite
Elements suffer from some serious drawbacks, such as the strong
difficulties in the three dimensional case (polyhedrons) and in
the use of non convex elements.
On the contrary, the VEM is free from the above-mentioned troubles, and thus it represents a very promising approach for Computational Structural Mechanics problems.

Aim of the present paper is to initiate the investigation on the
VEM when applied to non-linear elastic and
inelastic problems in small deformations. More precisely, we
mainly focus on the following cases: 1) non-linear elastic
constitutive laws in a small deformation regime which, however,
pertain to stable materials; 2) inelastic constitutive laws in a
small deformation regime as they arise, for instance, in classical
plasticity problems. We remark that we are not going to consider
here situations with internal constraints, such as
incompressibility, which require additional peculiar numerical
treatment. 
Virtual elements for the linear elasticity problem where
introduced in \cite{elasticity,Paulino-VEM}. The scheme in the
present paper is one of the very first developements of the VEM
technology for nonlinear problems, and it is structured in such a way
that a general non linear constitutive law can be automatically
included. Indeed, on every element of the mesh the constitutive
law needs only to be applied once (similarly to what happens in one-point Gauss
quadrature scheme) and the constitutive law algorithm can be independently embedded
 as a self-standing black-box, as in common
engineering FEM schemes. Therefore, in addition to the advantage
of handling general polygons/polyhedra, the present method is
computationally efficient, in the sense that the constitutive law
needs to be applied only once per element at every iteration
step. The risk of ensuing hourglass modes is avoided by using an
evolution of the standard VEM stabilization procedure used in
linear problems.
However, we highlight that the proposed method is described for general $d$-dimensional problems ($d=2,3$), but the performed numerical experiments are confined to the two dimensional setting.

A brief outline of the paper is as follows. In Section~\ref{sec:1}
we describe the continuous problems we are interested in. In
particular, we distinguish between the elastic, possibly
non-linear, case (Section~\ref{sec:1.1}), and the general
inelastic case (Section~\ref{sec:1.2}). Section~\ref{sec:2} deals
with the VEM discretization. After having introduced the
approximation spaces and the necessary projection operators
(Section~\ref{sec:2.1}), we detail the discrete problems for the
elastic case in Section~\ref{sec:2.2}, and for the inelastic case
in Section~\ref{sec:2.3}. In Section~\ref{theor-results}, combining ideas and techniques from~\cite{Ciarlet:78} and~\cite{volley}, we
provide some theoretical results concerning the convergence of the
proposed scheme in the elastic situation. We remark that our
analysis is confined to cases where the non-linear costitutive law
fulfills suitable continuity and stability properties, as stated
at the beginning of the Section. Section~\ref{sec:num} presents
several numerical examples which asses the actual behaviour of the
proposed scheme. In Sections~\ref{sec:test1} and~\ref{sec:test2}
we consider non-linear elastic cases, while in
Section~\ref{sec:plast} a von Mises plasticity problem with
hardening is detailed. Furthermore, an initial brief discussion
about a possible extension to large deformation problem is
included (Section~\ref{s:largedef}). Finally, we draw some
conclusion in Section~\ref{sec:concls}.

Throughout the paper, we will make use of standard notations regarding Sobolev spaces, norms and seminorms, see~\cite{Boffi-Brezzi-Fortin}, for instance. In addition, $C$ will denote a constant independent of the meshsize, not necessarily the same at each occurrence. Finally, given two real quantities $a$ and $b$, we will write $a\lesssim b$ to mean that there exists $C$ such that $a\le C b$.

\section{The continuous problems}
\label{sec:1}

In the present section we describe the problem considered in this paper. Although the elastic case could be considered as a particular instance of the inelastic case, we prefer to keep the presentation of the two problems separate. This will allow us a clearer presentation of the ideas of the virtual element scheme in the following section.

\subsection{The elastic case}
\label{sec:1.1}

We consider an elastic body $\Omega \subset {\mathbb R}^d$ ($d=2,3$) clamped on part $\Gamma$ of the boundary and subjected to a body load $\f$. We are interested, assuming a regime of small deformations, in finding the displacement $\u:\Omega \rightarrow {\mathbb R}^d$ of the deformed body.

We are given a constitutive law for the material at every point $x \in \Omega$, relating strains to stresses $\sigmab$, through the function
\begin{equation}\label{con-law}
\sigmab = \sigmab(x,\nabla \u(x)) \in {\mathbb R}^{d\times d}_{\textrm{symm}}
\end{equation}
where $\nabla \u$ represents the gradient of the displacement $\u$.

Given the law \eqref{con-law}, the deformation problem reads
\begin{equation}\label{elast-prob-strong}
\left\{
\begin{aligned}
& - \div\, \sigmab = \f  \qquad & \textrm{in} \ \Omega , \\
& \u = 0  \qquad & \textrm{on} \ \Gamma , \\
& \sigmab \n = 0 \qquad & \textrm{on} \ \partial\Omega / \Gamma ,
\end{aligned}
\right.
\end{equation}
where $\n$ denotes the unit outward normal to $\partial \Omega$.

Let now $\V$ denote the space of admissible displacements and $\W$ the space of its variations; both spaces will, in particular, satisfy the homogeneous Dirichlet boundary condition on $\Gamma$.
The variational formulation of the elastic deformation problem reads
\begin{equation}\label{elast-prob}
\left\{
\begin{aligned}
& \textrm{Find } \u \in \V \textrm{ such that } \\
& \int_{\Omega} \sigmab(x,\nabla \u(x)) : \nabla \v(x) \dx = \int_{\Omega} \f(x)\cdot\v(x) \dx
\quad \forall \v \in \W .
\end{aligned}
\right.
\end{equation}

\begin{remark}
The generalization of the results of the present paper to other type of loadings (for instance in the presence of boundary forces) and boundary conditions (for instance in the presence of enforced displacements) is trivial. Our choice in \eqref{elast-prob-strong} allows to keep the exposition shorter.
\end{remark}

\subsection{The inelastic case}
\label{sec:1.2}

We assume a small deformation regime and restrict ourselves to rate independent inelasticity.
We consider a material body $\Omega \subset {\mathbb R}^d$ ($d=2,3$) clamped on part $\Gamma$ of the boundary and subjected to a body load $\f(t,x)$ depending also on a pseudo-time variable $t \in [0,T]$.  The interested reader can find more details in \cite{HanReddy, SimoHughes}, for instance.
We are interested in finding the displacement $\u:\Omega \rightarrow {\mathbb R}^d$ of the deformed body at a given final time $T$.

We are given an inelastic constitutive law for the material, relating strains to stresses $\sigmab$, through the function
\begin{equation}\label{claw-ine}
\sigmab = \sigmab(x,\nabla \u(x), H_x) \in {\mathbb R}^{d\times d}_{\textrm{symm}}
\end{equation}
where the vector $H_x$ contains all history variables at the point $x$.

The above rule is to be coupled with an evolution law $\mathcal{L}$ for the history variables in time
\begin{equation}\label{evol-law}
\dot{\cal H}_x = {\cal L}(x,\nabla u(x),\dot{\nabla}u(x),{\cal H}_x) ,
\end{equation}
where, as usual, a dot above a function stands for a pseudo-time  derivative.
Since we consider a quasi-static problem, at each time instant the stresses and displacements must satisfy the equilibrium and boundary conditions in \eqref{elast-prob-strong}.

We here avoid to write a rigorous variational formulation for the problem above, and limit ourselves to the minimal setting that will be needed to introduce the associated discrete problem.
As in the elastic case, let $\V$ denote the space of admissible displacements and $\W$ the space of its variations.
Then, assuming an initial value for the history variables, the quasi-static inelastic deformation problem can be written as
\begin{equation}\label{inelast-prob}
\left\{
\begin{aligned}
& \textrm{For all } t \in [0,T] \textrm{ find } \u(t,\cdot) \in \V \textrm{ such that } \\
& \int_{\Omega} \sigmab(x,\nabla \u(t,x), H_x(t)) : \nabla \v(x) \dx =  \int_{\Omega} \f(t,x)\cdot\v(x) \dx
\quad \forall \v \in \W ,
\end{aligned}
\right.
\end{equation}
where the displacements and history variables are sufficiently regular in time and must satisfy the evolution law \eqref{evol-law} almost everywhere.

\section{The virtual element approximation}
\label{sec:2}

In the present section we introduce the virtual element discretization of problems \eqref{elast-prob} and \eqref{inelast-prob}. In what follows, given any subset $\omega$ of ${\mathbb R}^d$ ($d=2,3$) and $k \in \mathbb{N}$, we denote by $\Pp_k(\omega)$ (respectively $\Ppb_k(\omega)$) the scalar (respectively vector with $d$ components) polynomials of degree up to $k$ on $\omega$.

\subsection{The virtual spaces and operators}
\label{sec:2.1}

We consider a mesh $\Omega_h$ for the domain $\Omega$, made of general polygonal/polyhedral conforming elements. For the time being, we only assume that such mesh is compatible with the boundary conditions, i.e. that $\Gamma$ is union of faces (edges) of the mesh.
We denote by $E \in \Omega_h$ the generic element of the mesh and by $f$ the generic face (or edge if $d=2$).
The symbols $h_E$ and $| E |$ will represent, respectively, diameter and volume (or area) of the element $E$. As usual, $h$ will indicate the maximum element size.

We start by introducing the discrete virtual space for displacements, that is essentially the same as in \cite{elasticity}.
We first consider the two dimensional case. Given any $E\in\Omega_h$, let the local virtual space
\begin{equation}\label{2d-VhE}
\VhE : = \big\{ \v \in [H^1(E) \cap C^0(E)]^2 \: : \: \Delta\v = 0 \textrm{ in } E, \: \v|_f \in \Ppb_1(f) \: \forall f \in \partial E \big\} ,
\end{equation}
where $\Delta$ denotes the component-wise Laplace operator. The space $V_{h,E}$ is a space of harmonic functions that on the boundary of the element are piecewise linear (edge by edge) and continuous. Such space is \emph{virtual} in the sense that is well defined but not known explicitly inside the element.

Note that $ \Ppb_1(E) \subseteq\VhE$; in the case of a triangular element, we recover exactly the standard $\Ppb_1$ space. It is easy to check \cite{elasticity} that a set of degrees of freedom for the space $\VhE$ is simply given by the collection of the vertex values:
$$
\bullet \ \textrm{Pointwise values } \{ \v(\nu) \}_{\nu \in \partial E} \textrm{ with } \nu \textrm{ denoting a vertex of } E .
$$
Once the above degrees of freedom values are given, since $\v \in \VhE$ is linear on each edge, the value of $\v$ on the boundary $\partial E$ is completely determined. Therefore, an integration by parts allows to compute the integral average of the gradient
\begin{equation}\label{int-grad}
\frac{1}{|E|} \int_E \nabla \v \dx = \frac{1}{|E|} \sum_{f \in\partial E} \int_f \v \otimes \n_f \ds \quad \forall\, \v \in \VhE ,
\end{equation}
with $\n_f$ indicating the outward unit normal at each edge $f$.

\medskip

We now define the virtual local spaces for the three dimensional case. Given a polyhedron $E \in \Omega_h$, any face $f \in \partial E$ is now a polygon. We denote by $\Vhf$ the virtual bi-dimensional space \eqref{2d-VhE} on the polygon $f$ adjusted with three components:
\begin{equation}\label{Vhf}
\Vhf : = \big\{ \v \in [H^1(f) \cap C^0(f)]^3 \: : \: \Delta\v = 0 \textrm{ in } f, \: \v|_e \in \Ppb_1(e) \: \forall e \in \partial f \big\} ,
\end{equation}
where the symbol $e$ represents the generic edge of the polyhedron and $\Delta$ denotes the planar laplacian on $f$.
We then define
\begin{equation}\label{3d-VhE}
\VhE = \big\{ \v \in [H^1(E) ]^3 \: : \: \Delta\v = 0 \textrm{ in } E, \: \v|_f \in \Vhf \: \forall f \in \partial E \big\} .
\end{equation}
The space $V_{h,E}$ is a space of harmonic functions that on the boundary of the element are continuous and, on each face, functions of $\Vhf$. Note that, as a consequence, the functions of $\VhE$ are linear on each edge of the polyhedron.

Again we note that $ \Ppb_1(E) \subseteq\VhE$; in the case of a tetrahedral element, we recover exactly the standard $\Ppb_1$ space. It is easy to check that a set of degrees of freedom for the space $\VhE$ is again given by
$$
\bullet \ \textrm{Pointwise values } \{ \v(\nu) \}_{\nu \in \partial E} \textrm{ with } \nu \textrm{ denoting a vertex of } E .
$$
An integration by parts exactly as in \eqref{int-grad} allows to compute, for all $E \in \Omega_h$ the integral average of the gradient, provided one is able to compute the face integrals $ \int_f \v \otimes \n_f \ds $ for all $ f\in \partial E $ and $ \v \in \VhE $.
Such face integrals can be easily computed by introducing the virtual space modification proposed in \cite{enhanced}, that we do not detail here. The result is
$$
\int_f \v \otimes \n_f = \sum_{\nu \in \partial E} \omega_\nu \v (\nu)  ,
$$
where the scalars $\{ \omega_\nu \}_{\nu \in \partial E}$ are the weights of any integration rule on the face that is exact for linear functions.

\medskip

Once the local virtual spaces are defined, all that follows holds identically in two and three dimensions.
We can now present the global virtual space
$$
\Vh := \big\{ \v \in \V \: : \: \v|_{E} \in \VhE \quad \forall E \in \Omega_h \big\} .
$$
A set of degrees of freedom for $\Vh$ is given by all pointwise values of $\v$ on all vertices of $\Omega_h$, excluding the vertices on $\Gamma$ (where the value vanishes).

In the following, we will denote by $\Pi^0$ the tensor valued $L^2$ projection operator on the space of piecewise constant functions and by $\Pi^0_E$ its restriction to the generic element $E \in \Omega_h$. More precisely, for any ${\bf G} \in (L^2(\Omega))^{d \times d}$, we have $(\Pi^0 {\bf G} )_E = \Pi^0_E ({\bf G} |_E)$ with
the local operators defined as
\begin{equation}\label{AL-proj1}
\Pi^0_E {\bf G} |_E = \frac{1}{|E|} \int_E {\bf G}  \dx \qquad \forall E \in \Omega_h .
\end{equation}
We have the following important remark, which is a direct consequence of~\eqref{int-grad}.
\begin{remark}\label{rem:comp}
For all functions $\v\in\VhE$ and all elements $E\in\Omega_h$, the operators
$\Pi^0_E(\nabla \v)$ are explicitly computable.
\end{remark}

We moreover introduce a second projection operator $\Pi^\nabla$, defined on $\Vh$ as follows. For any $\v \in \Vh$, we have $(\Pi^\nabla \v)_E = \Pi^\nabla_E (\v|_E)\in \Ppb_1(E)$ with
the local operators defined as
\begin{equation}\label{AL-proj2}
\left\{
\begin{aligned}
& \nabla (\Pi^\nabla_E( \v |_E)) = \Pi^0_E (\nabla \v|_E) , \\
&  \sum_{\nu \in \partial E} (\Pi^\nabla_E \v) (\nu) = \sum_{\nu \in \partial E} \v(\nu)
\end{aligned}
\right.
\end{equation}
for all $E$ in $\Omega_h$.
Note that, by definition, $\Pi^\nabla \v$ is a (discontinuous) piecewise linear function on $\Omega_h$. On each element  $E$, $\Pi^\nabla_E( \v |_E)$  is the unique linear function such that:

\begin{enumerate}

\item its (constant) gradient equals the mean value over $E$ of the function $\nabla \v$;

\item its vertex value average equals the vertex value average of  $\v$.

\end{enumerate}

We notice that the second condition in~\eqref{AL-proj2} is only to fix the constant part of $\Pi^\nabla \v$ on each element.
Recalling Remark \ref{rem:comp}, it is immediate to check that the operator $\Pi^\nabla$ is explicitly computable.

\subsection{The elastic case}
\label{sec:2.2}

The main missing step is to introduce the local forms that will be used in the discrete variational formulation.
We assume that the constitutive law \eqref{con-law} is piecewise constant with respect to the mesh $\Omega_h$. Therefore, instead of $\sigmab(x,\nabla \u(x))$, we will write $\sigmab_E(\nabla \u(x))$ to represent the constitutive law on $E$, $E \in \Omega_h$ and $x \in E$. In addition, foe every pair $\v\in\V$ and $\w\in \W$, we introduce the bilinear forms $a_E(\v, \w)$ and $a(\v,\w)$ as:

\begin{equation}\label{a-forms}
\begin{aligned}
a_E(\v,\w)& = \int_E  \sigmab(x,\nabla \v(x)) : \nabla \w(x) \dx , \\
a(\v,\w)& = \int_\Omega  \sigmab_E(\nabla \v(x)) : \nabla \w(x) \dx .
\end{aligned}
\end{equation}
Therefore, it holds
\begin{equation}\label{local-sum}
a(\v,\w) =\sum_{E\in\Omega_h} a_E(\v,\w)
\end{equation}
and, recalling~\eqref{elast-prob}, the elastic problem can be written as

\begin{equation}\label{elast-prob-short}
\left\{
\begin{aligned}
& \textrm{Find } \u \in \V \textrm{ such that } \\
& a(\u,\v) = \int_{\Omega} \f(x)\cdot\v(x) \dx
\quad \forall \v \in \W .
\end{aligned}
\right.
\end{equation}

We now consider, for all $E \in \Omega_h$ and all $\v_h,\w_h \in \VhE$, the following \emph{preliminary} form
\begin{equation}\label{a-tilde}
\begin{aligned}
\athE(\v_h,\w_h) &=  \int_E \sigmab_E(\Pi^0_E (\nabla \v_h)(x)) : (\Pi^0_E (\nabla \w_h)(x)) \dx \\
& = |E| \: \sigmab_E(\Pi^0_E (\nabla \v_h)) : \Pi^0_E (\nabla \w_h) ,
\end{aligned}
\end{equation}
where the identity above follows since all the involved functions are constant on the element.
The above form is $\Ppb_1$-consistent, in the sense that it recovers exactly the original form whenever the first entry is a linear polynomial. Indeed, it follows from~\eqref{AL-proj1} and~\eqref{AL-proj2} that

\begin{equation}\label{a-tilde-cons}
\begin{aligned}
\athE(\q,\v_h) &=  \int_E \sigmab_E(\Pi^0_E (\nabla \q)(x)) : (\Pi^0_E (\nabla \w_h)(x)) \dx   \\
&  =  \int_E \sigmab_E(\nabla \q(x)) : (\Pi^0_E (\nabla \w_h)(x)) \dx   =  \int_E \sigmab_E(\nabla \q(x)) : \nabla \w_h(x) \dx  \\
& = a_E (\q,\v_h) \qquad \forall \q\in \Ppb_1(E) , \forall \v_h \in \VhE .
\end{aligned}
\end{equation}

However, unless the elements are triangular/tetrahedral, the form $\athE(\cdot,\cdot)$ has a non-physical kernel that may lead to spurious modes in the solution.
We therefore follow the idea proposed initially in \cite{volley} and introduce the discrete bilinear form
\begin{equation}\label{Sh}
\begin{aligned}
& S_{h,E} \: : \: \VhE \times \VhE \: \longrightarrow {\mathbb R} , \\
& S_{h,E}(\v_h,\w_h) = h_E^{d-2} \: \sum_{\nu \in \partial E} \v_h(\nu) \w_h(\nu) \qquad \forall \v_h,\w_h \in \VhE .
\end{aligned}
\end{equation}

As discussed in \cite{elasticity,volley}, under suitable mesh regularity assumptions detailed in Section \ref{theor-results}, there exist positive constants $c_*,c^*$ independent of the element such that
\begin{equation}\label{equiS}
c_* \int_E || \nabla^{\rm sym} \v_h ||^2  \dx   \le S_{h,E}(\v_h,\v_h) \le c^* \int_E || \nabla^{\rm sym} \v_h ||^2  \dx
\end{equation}
for all $\v_h \in \VhE$ with $\Pi^\nabla_E \v_h=0$.
In other words, on the orthogonal complement of $\Ppb_1(E)$ with respect to $\VhE$, the bilinear form $S_{h,E}(\cdot, \cdot)$ behaves as the local energy of a linearly elastic body with unitary material constants and is thus suitable to stabilize $\athE(\cdot,\cdot)$ form in such case.
In order to take into account different material constants and also nonlinear materials, the form $S_{h,E}(\cdot, \cdot)$ needs to be multiplied by a positive constant $\alpha_E$ that may depend on the discrete solution.

We therefore introduce the following local virtual form on $\VhE$.
For all $E \in \Omega_h$ and all $\s_h, \v_h,\w_h \in \VhE$
\begin{equation}\label{a-local}
\begin{aligned}
\ahE(\s_h;\v_h,\w_h) & =  \athE(\v_h,\w_h) + \alpha_E(\s_h) S_{h,E}(\v_h - \Pi^\nabla_E \v_h, \w_h - \Pi^\nabla_E \w_h) ,
\end{aligned}
\end{equation}
where the {\em stabilizing parameter} $\alpha_E>0$ depends on the additional entry $\s_h$. We remark that the bilinear form $\ahE(\cdot; \cdot ,\cdot)$ is still $\Ppb_1$-consistent. This follows from~\eqref{a-tilde-cons} and the observation that  $\q - \Pi^\nabla_E \q = 0$ for every $ \q\in\Ppb_1 $.
The choice that we here propose for the parameter $\alpha_E$ is
\begin{equation}\label{alfa}
\alpha_E(\s_h) = ||| \frac{\partial \sigmab_{\! E}}{\partial \nabla \u}(\Pi^0_E \nabla \s_h|_E) |||
\qquad \forall E \in \Omega_h, \ \forall \s_h \in \Vh ,
\end{equation}
with $||| \cdot |||$ representing any norm on the fourth order tensor space, for instance the maximum of the absolute values of all the entries, see Remark~\ref{rem:alfa}.

We present also the global form

\begin{equation}\label{a-global}
\ah(\s_h;\v_h,\w_h) = \sum_{E\in\Omega_h} \ahE(\s_h;\v_h,\w_h) \quad \forall \s_h, \v_h,\w_h \in \Vh .
\end{equation}

Given $\s_h\in \Vh$, a possible virtual discretization of Problem \eqref{elast-prob} is
\begin{equation}\label{elast-virt-prob-prel}
\left\{
\begin{aligned}
& \textrm{Find } \u_h \in \Vh \textrm{ such that } \\
&  \ah(\s_h;\u_h,\v_h)  = <\f , \v_h>_h \quad \forall \v_h \in \Vh .
\end{aligned}
\right.
\end{equation}
Above, the load approximation term
$$
<\f , \v_h>_h = \sum_{\nu \in \partial E} \omega_\nu \f(\nu) \v (\nu)
$$
is a vertex-based quadrature rule with weights $\omega_\nu$  chosen to provide the exact integral on $E$ when applied to linear functions. Furthermore, a reasonable choice for $\s_h$ could be $\s_h=\u_h$.

We instead propose a modification of~\eqref{elast-virt-prob-prel}, that is more practical from the implementation viewpoint.
We assume the usual incremental loading procedure for the solution of the nonlinear discrete problem: given a positive integer $N$, let the partial loadings $\f^n = (n/N) \f$ for all $n=1,2,...,N$.
Then, given the initial displacement $\u_h^0$ (for instance the zero function), one applies for $n=1,2,...,N$ the iterative procedure
\begin{equation}\label{elast-virt-prob}
\left\{
\begin{aligned}
& \textrm{Find } \u_h^n \in \Vh \textrm{ such that } \\
&  \ah(\u_h^{n-1};\u_h^n,\v_h)  = <\f^n , \v_h>_h \quad \forall \v_h \in \Vh .
\end{aligned}
\right.
\end{equation}
The final solution is then $\u_h=\u_h^N$.
The nonlinear problems above can be solved with the Newton scheme. Note that, since the stability constants $\alpha_E$ (see \eqref{a-local}) are computed by using $\u_h^{n-1}$, the tangent matrix in the Newton iterations turns out to be simpler. Since $N$ is typically taken large (at least 10, but often much more) the effect of such modification is not detrimental for the discrete approximation; the constants $\alpha_E$ are only used as scaling parameters and do not enter the accuracy of the algorithm.

We close the section with some observations regarding the local forms $\ahE$ used in the scheme. First, we recall that the proposed forms are  $\Ppb_1(E)$-consistent, in the sense that for all $E\in\Omega_h$, we have:
\begin{equation}\label{consistency}
\ahE(\s_h;{\bf q},\v_h) = \int_E \sigma_E(\nabla {\bf q}) : \nabla \v_h \dx
\quad \forall \s_h, \v_h \in \VhE , \ \forall {\bf q} \in \Ppb_1(E) .
\end{equation}
Identity \eqref{consistency} is a fundamental condition for approximation and, in particular, guarantees the satisfaction of the patch test.
Moreover, such forms are explicitly computable for any polygonal/polyhedral element (even non-convex). Finally, the constitutive law needs to be computed only \emph{once} per element and thus the method, from this point of view, is as cheap as finite elements with one point gauss integration rule. This observation has an even bigger impact in the inelastic case, where the constitutive laws are typically more expensive to compute.

\begin{remark}\label{rem:alfa}
The motivation for choice \eqref{alfa} and~\eqref{elast-virt-prob} is to better mimic the stability properties of the material for the current displacement.
For materials in which the stress-strain incremental relation does not depend too strongly on the value of the current displacement, the constants $\alpha_E$ can be taken as independent of $\u_h^{n-1}$. For instance, a scaling directly proportional to the local material constants could be used. On the other hand, the choice proposed in \eqref{alfa} and~\eqref{elast-virt-prob} give good results for a wider range of materials. Examples and investigations in this direction can be found in Section~\ref{sec:num}.
\end{remark}

\subsection{The inelastic case}
\label{sec:2.3}

We start by introducing a sub-division of the ``time'' interval $[0,T]$ into smaller intervals $[t_{n-1},t_n]$ for $n=1,2,...,N$, where for simplicity we assume that $t_n = nT/N$. We will denote the partial loadings by $\f^n = (n/N) \f$ for all $n=1,2,...,N$.

We assume, as in standard engineering procedures, a constitutive algorithm that is an approximation of the constitutive and evolution laws \eqref{claw-ine}, \eqref{evol-law}. In Finite Element analysis, this pointwise algorithm can be coded independently from the global FE construction and can be regarded as a ``black-box'' procedure that is applied at every Gauss point and at every iteration step.
In the present Virtual Element method, we want to keep the same approach; in other words, our scheme will be compatible with any black-box constitutive algorithm that follows in the general setting below and that can be imported from other independent sources.

We assume that the constitutive law is piecewise constant with respect to the mesh $\Omega_h$.
Let $\hat\sigmab_E$ represent the constitutive algorithm for the element $E \in \Omega_h$. For any $x \in E$, given a value for the displacement gradient $\nabla\u_h^{n-1}(x)$ at time $t_{n-1}$, a value $H_x^{n-1}$ for the history variables at time $t_{n-1}$ and a tentative value for the displacement gradient $\nabla\u_h^n(x)$ at time $t_n$, the algorithm computes the stresses (and updates the history variables) at time $t_n$.
We thus write the computed stress as
$$
\widehat\sigmab_E (\nabla \u_h^{n-1}(x), H_x^{n-1}, \nabla \u_h^n(x)) .
$$
As part of the approximation procedure of our method, we assume that the history variables $H_x$ are piecewise constant with respect to the mesh, and therefore write $H_E^n$ to represent the value assumed on the element $E\in\Omega_h$ at time $t_n$. Consistently, $H^n$ will represent the collection of all $\{ H_E^n \}_{E \in \Omega_h}$.

In our scheme, instead of applying the constitutive algorithm at Gauss points, we make use of the projections introduced in the previous sections and of the same stabilization as in the elastic case.
The Virtual Element scheme reads, for $n=1,2,...,N$:
\begin{equation}\label{inelast-virt-prob}
\left\{
\begin{aligned}
& \textrm{Find } \u_h^n \in \Vh \textrm{ (and the updated } H^n) \textrm{ such that } \\
& \ah(\u_h^{n-1},\u_h^n,H^{n-1},\v_h)  = <\f^n , \v_h>_h \quad \forall \v_h \in \Vh ,
\end{aligned}
\right.
\end{equation}
where the form
$$
\ah(\u_h^{n-1},\u_h^n,H^{n-1},\v_h) = \sum_{E\in\Omega_h} \ahE(\u_h^{n-1},\u_h^n,H_E^{n-1},\v_h)
$$
with, for all $E \in \Omega_h$,
$$
\begin{aligned}
\ahE(\u_h^{n-1},\u_h^n,H_E^{n-1},\v_h) = &
|E| \: \widehat\sigmab_E(\Pi^0_E \nabla \u_h^{n-1}, H_E^{n-1}, \Pi^0_E \nabla \u_h^n) : \Pi^0_E (\nabla \v_h) \\
& + \alpha_E(\u_h^{n-1}) S_{h,E}(\u_h^n - \Pi^\nabla_E \u_h^n, \v_h - \Pi^\nabla_E \v_h) .
\end{aligned}
$$
Here above, the bilinear form $S_{h,E}$ and the scalar $\alpha_E$ are calculated as already shown in
\eqref{Sh} and \eqref{alfa}, respectively.
Note that, as already mentioned in Section \ref{sec:2.2}, the constitutive algorithm needs to be applied only once per element.



\section{Theoretical results}\label{theor-results}

We here develop an error analysis for the method described in Section~\ref{sec:2.2}, under some additional hypotheses on the function $\sigmab(x,\nabla \u(x))=\sigmab_{\! E}(\nabla\u(x))$. More precisely, we assume that the following properties are satisfied.

\medskip

\noindent {\bf Hypotheses (RPC)}

\begin{itemize}

\item The function $\taub \mapsto \sigmab_{\! E}(\taub) $ belongs to ${\bf C}^1({\mathbb R}^{d\times d})$ for every $E\in\Omega_h$;

\item for every $E\in\Omega_h$, the differential $\frac{\partial \sigmab_{\! E}}{\partial \taub}(\taub)$ satisfies

\begin{enumerate}

\item there exists $C_\alpha>0$ such that

\begin{equation}
\frac{\partial \sigmab_{\! E}}{\partial \taub}(\taub)\, \s : \s \ge C_\alpha ||\s||^2\qquad \forall\, \s \in {\mathbb R}^{d\times d} ,
\label{sigmaposit}
\end{equation}

\item there exists $C_M>0$ such that

\begin{equation}
\frac{\partial \sigmab_{\! E}}{\partial \taub}(\taub)\, \s : \t \le C_M ||\s|| \,  ||\t||\qquad \forall\, \s ,\t \in {\mathbb R}^{d\times d} .
\label{sigmacont}
\end{equation}

\end{enumerate}

\end{itemize}

We moreover explicit here the shape regularity conditions that are needed for the theoretical results of the present paper. We assume that there exists a positive constant $C_s$ such that all the elements $E$ of the mesh sequence are star shaped with respect to a ball of radius $\rho \ge C_s h_E$ and that all the edges $e$ of $E$ have length $h_e \ge C_s h_E$.

\begin{lemma}
Let the bilinear forms $a_E(\cdot,\cdot)$,  $a(\cdot,\cdot)$, $\ahE(\cdot;\cdot,\cdot)$ and  $\ah(\cdot; \cdot,\cdot)$ be defined by~\eqref{a-forms}, \eqref{a-local} and~\eqref{a-global}. Suppose that the Hypotheses (RPC) introduced above are satisified. Then, it holds
\begin{equation}
| \v_h - \w_h |_{1,\Omega}^2 \lesssim \ah( \s_h;\v_h, \v_h - \w_h)  - \ah(\s_h;\w_h, \v_h - \w_h) \qquad \forall\, \v_h,\w_h,\s_h\in \Vh .
\label{coerciv-est}
\end{equation}
\begin{equation}
 a_E( \v, \r )  - a_E(\w, \r )
\lesssim | \v -\w  |_{1,E} | \r|_{1,E}\qquad \forall\, \v ,\w , \r \in \V .
\label{cont-est}
\end{equation}
\begin{equation}
 \ahE( \s_h;\v_h, \r_h)  - \ahE(\s_h;\w_h, \r_h)
\lesssim | \v_h -\w_h  |_{1,E} | \r_h |_{1,E}\qquad \forall\, \v_h,\w_h,\s_h,\r_h\in \Vh .
\label{cont-esth}
\end{equation}
\end{lemma}
\begin{proof}
We first note that~\eqref{sigmaposit} and~\eqref{sigmacont}, together with~\eqref{alfa}, imply the existence of positive constants $c_1$ and $c_2$ such that
\begin{equation}\label{alfa-bounds}
c_1 \le \alpha_E(\s_h) \le c_2
\qquad \forall E \in \Omega_h, \ \forall \s_h \in \Vh .
\end{equation}

{\it Step (i): proof of~\eqref{coerciv-est}.} From~\eqref{sigmaposit}, we deduce that

\begin{equation}
\left(
\sigmab_{\! E}(\s) - \sigmab_{\! E}(\t) \right): \left( \s - \t \right)
\ge C ||\s -\t ||^2\qquad \forall\, \s , \t \in {\mathbb R}^{d\times d} .
\label{lemma1}
\end{equation}
Therefore, for every $\v,\w \in \V$ we have
\begin{equation}
a_E(\v, \v-\w)- a_E(\w, \v-\w)= \int_E\left( \sigmab_{\! E}(\nabla\v) -\sigmab_{\! E}(\nabla\w) \right):(\nabla\v -\nabla\w)
\ge C | \v-\w  |_{1,E}^2 ,
\label{lemma2}
\end{equation}
by which
\begin{equation}
 | \v-\w  |_{1,\Omega}^2 \lesssim a(\v, \v-\w)- a(\w, \v-\w)  \qquad \forall\, \v,\w \in \V  .
\label{lemma3}
\end{equation}
For every $\v_h,\w_h , \s_h \in \Vh$, we have (see~\eqref{a-local})
\begin{equation}
\begin{aligned}
\ahE(& \s_h;\v_h, \v_h - \w_h)  - \ahE(\s_h;\w_h, \v_h - \w_h)\\
&=  \athE(\v_h,\v_h - \w_h) -  \athE(\w_h,\v_h - \w_h) \\
& + \alpha_E(\s_h) S_{h,E}( (\v_h -\w_h) - \Pi^\nabla_E( \v_h  -\w_h),  (\v_h  - \w_h ) - \Pi^\nabla_E( \v_h - \w_h)) .
\end{aligned}
\label{lemma4}
\end{equation}
We now notice that (see~\eqref{a-tilde})
\begin{equation}
\begin{aligned}
&\athE(\v_h,\v_h - \w_h) -  \athE(\w_h,\v_h - \w_h)\\
&\qquad = \int_E \sigmab_E(\Pi^0_E (\nabla \v_h)) : (\Pi^0_E( \nabla \v_h) - \Pi^0_E( \nabla \w_h)) \\
&\qquad\qquad -
\int_E \sigmab_E(\Pi^0_E (\nabla \w_h)) : (\Pi^0_E( \nabla \v_h) - \Pi^0_E( \nabla \w_h))\\
&\qquad =
\int_E\big[ \sigmab_E(\Pi^0_E (\nabla \v_h)) -  \sigmab_E(\Pi^0_E (\nabla \w_h))\big] : (\Pi^0_E( \nabla \v_h) - \Pi^0_E( \nabla \w_h)) .
\end{aligned}
\label{lemma5}
\end{equation}
First using~\eqref{lemma1} with $\s = \Pi^0_E (\nabla \v_h)$ and $\t = \Pi^0_E (\nabla \w_h)$, then recalling \eqref{AL-proj2} we get
\begin{equation}
\begin{aligned}
\athE(\v_h,\v_h - \w_h) & -  \athE(\w_h,\v_h - \w_h)\ge C
||\Pi^0_E (\nabla \v_h)  -\Pi^0_E (\nabla \w_h) ||_{0,E}^2 \\
&= C\,  ||    \nabla (\Pi^\nabla_E ( \v_h -\w_h )) ||_{0,E}^2 =
C\, | \Pi^\nabla_E ( \v_h -\w_h ) |_{1,E}^2 .
\end{aligned}
\label{lemma6}
\end{equation}
In addition, we have, using~\eqref{alfa-bounds} and~\eqref{equiS}:
\begin{equation}
\begin{aligned}
\alpha_E(\s_h) & S_{h,E}( (\v_h -\w_h) - \Pi^\nabla_E( \v_h  -\w_h),  (\v_h  - \w_h ) - \Pi^\nabla_E( \v_h - \w_h)) \\
& \ge C |   (\v_h -\w_h) - \Pi^\nabla_E( \v_h  -\w_h)  |_{1,E}^2
\end{aligned}
\label{lemma7}
\end{equation}
Combining~\eqref{lemma4} with~\eqref{lemma6} and~\eqref{lemma7}, we infer
\begin{equation}
| \v_h - \w_h |_{1,E}^2 \lesssim \ahE( \s_h;\v_h, \v_h - \w_h)  - \ahE(\s_h;\w_h, \v_h - \w_h) \qquad \forall\, \v_h,\w_h,\s_h\in \Vh .
\label{lemma8}
\end{equation}
Summing up over all the elements, we get~\eqref{coerciv-est}:
\begin{equation}
| \v_h - \w_h |_{1,\Omega}^2 \lesssim \ah( \s_h;\v_h, \v_h - \w_h)  - \ah(\s_h;\w_h, \v_h - \w_h) \qquad \forall\, \v_h,\w_h,\s_h\in \Vh .
\label{lemma9}
\end{equation}

{\it Step (ii): proof of~\eqref{cont-est} and~\eqref{cont-esth}.} From~\eqref{sigmacont}, we deduce that
\begin{equation}
\left(
\sigmab_{\! E}(\s) - \sigmab_{\! E}(\t) \right): \taub
\le C || \s -\t  || \,  || \taub ||
\qquad \forall\, \s , \t , \taub \in {\mathbb R}^{d\times d} ,
\label{lemma10}
\end{equation}
%
%
%
from which we easily get~\eqref{cont-est}:
\begin{equation}
 a_E( \v, \r )  - a_E(\w, \r )
\lesssim | \v -\w  |_{1,E} | \r|_{1,E}\qquad \forall\, \v ,\w , \r \in \V .
\label{lemma12}
\end{equation}
We now notice that (see~\eqref{a-tilde})
\begin{equation}
\begin{aligned}
\athE(\v_h,\r_h) -  \athE(\w_h,\r_h) =
\int_E\big[ \sigmab_E(\Pi^0_E (\nabla \v_h)) -  \sigmab_E(\Pi^0_E (\nabla \w_h))\big] : \Pi^0_E( \nabla \r_h) .
\end{aligned}
\label{lemma13}
\end{equation}
Using~\eqref{lemma10}, identity \eqref{lemma13} yields
\begin{equation}
\begin{aligned}
\athE(\v_h,\r_h) -  \athE(\w_h,\r_h) \lesssim   |\v_h -\w_h |_{1,E} |\r_h |_{1,E} \qquad \forall\, \v_h , \w_h , \r_h \in \Vh.
\end{aligned}
\label{lemma15}
\end{equation}
To continue, since $S_{h,E}(\cdot,\cdot)$ is a bilinear form and using continuity arguments, we have for every $\s_h\in\Vh$ (see~\eqref{equiS})
\begin{equation}
\begin{aligned}
\alpha_E(\s_h)  S_{h,E}( \v_h  &- \Pi^\nabla_E( \v_h)  ,  \r_h- \Pi^\nabla_E( \r_h)) -
\alpha_E(\s_h)  S_{h,E}( \w_h  - \Pi^\nabla_E( \w_h)  ,  \r_h- \Pi^\nabla_E( \r_h))\\
 & = \alpha_E(\s_h)  S_{h,E}( (\v_h -\w_h) - \Pi^\nabla_E( \v_h  -\w_h),  \r_h- \Pi^\nabla_E( \r_h)) \\
& \lesssim |\v_h -\w_h |_{1,E} |\r_h |_{1,E} .
\end{aligned}
\label{lemma16}
\end{equation}
From~\eqref{a-local}, using~\eqref{lemma15} and~\eqref{lemma16}, we deduce~\eqref{cont-esth}.
%
%
\end{proof}

\begin{theorem}
Let $\u\in\V$ be the solution of Problem~\eqref{elast-prob}. Given any $\s_h\in \Vh$, let  $\u_h\in \Vh$ be the solution of Problem~\eqref{elast-virt-prob-prel}:
\begin{equation}\label{vp2}
\left\{
\begin{aligned}
& \textrm{Find } \u_h \in \Vh \textrm{ such that } \\
&  \ah(\s_h;\u_h,\v_h)  = <\f , \v_h>_h \quad \forall \v_h \in \Vh .
\end{aligned}
\right.
\end{equation}
For any $\u_I\in\Vh$ and $\u_\pi\in {\bf L}^2(\Omega)$ such that $\u_{\pi | E} \in \Ppb_1(E)$, it holds:
\begin{equation}
|\u   - \u_h |_{1,\Omega} \lesssim \sup_{\v_h\in \Vh} \frac{< \f , \v_h>_h - (\f,\v_h)}{|\v_h|_{1,\Omega}}   +
|\u - \u_I  |_{1,\Omega}   +    |\u - \u_\pi  |_{1,\Omega} ,
\label{errest}
\end{equation}
where $(\cdot,\cdot)$ denotes the $[L(\Omega)]^d$-scalar product.
\label{error-estimate}
\end{theorem}
\begin{proof}
Given $\u_I\in \Vh$, we set $\deltab_h= \u_h - \u_I$. For every $\u_\pi\in {\bf L}^2(\Omega)$ such that $\u_{\pi | E} \in \Ppb_1(E)$, using \eqref{coerciv-est} we have
\begin{equation}
\begin{aligned}
| & \u_h  - \u_I |^2_{1,\Omega} \lesssim
 a_h( \s_h; \u_h, \deltab_h)  -a_h(\s_h;  \u_I , \deltab_h)\\
&= < \f , \deltab_h>_h - \sum_{E\in\Omega_h}  \ahE(\s_h; \u_I ,\deltab_h)\\
& =
< \f , \deltab_h>_h - \sum_{E\in\Omega_h}  \Big\{
\big[\ahE(\s_h; \u_I ,\deltab_h)  - \ahE(\s_h; \u_\pi ,\deltab_h) \big] +
\ahE(\s_h; \u_\pi ,\deltab_h)   \Big\} .
\label{est1}
\end{aligned}
\end{equation}
Since \eqref{consistency} implies $\ahE(\s_h; \u_\pi ,\deltab_h) = a_E(\u_\pi ,\deltab_h)$, from~\eqref{est1} we get
\begin{equation}
\begin{aligned}
|\u_h  & - \u_I |^2_{1,\Omega} \lesssim\quad
< \f , \deltab_h>_h - \sum_{E\in\Omega_h}  \Big\{
\big[\ahE(\s_h; \u_I ,\deltab_h)  - \ahE(\s_h; \u_\pi ,\deltab_h) \big] +
a_E(\u_\pi ,\deltab_h)   \Big\} \\
& = < \f , \deltab_h>_h - \sum_{E\in\Omega_h}
\big[\ahE(\s_h; \u_I ,\deltab_h)  - \ahE(\s_h; \u_\pi ,\deltab_h) \big] \\
&\qquad\qquad   - \sum_{E\in\Omega_h} \big[
a_E(\u_\pi ,\deltab_h) -  a_E(\u ,\deltab_h) \big] - a(\u,\deltab_h) \\
& = \big[ < \f , \deltab_h>_h - (\f,\deltab_h) \big] - \sum_{E\in\Omega_h}
\big[\ahE(\s_h; \u_I ,\deltab_h)  - \ahE(\s_h; \u_\pi ,\deltab_h) \big] \\
&\qquad\qquad   - \sum_{E\in\Omega_h} \big[
a_E(\u_\pi ,\deltab_h) -  a_E(\u ,\deltab_h) \big] .
\label{est2}
\end{aligned}
\end{equation}
We then obtain, using \eqref{cont-est} and \eqref{cont-esth}
\begin{equation}
|\u_h   - \u_I |^2_{1,\Omega} \lesssim \left( \sup_{\v_h\in \Vh} \frac{< \f , \v_h>_h - (\f,\v_h)}{|\v_h|_{1,\Omega}}   +
|\u_I - \u_\pi  |_{1,\Omega}   +    |\u_\pi - \u  |_{1,\Omega}\right)  | \deltab_h |_{1,\Omega} ,
\label{est3}
\end{equation}
by which, recalling that $\deltab_h = \u_h -\u_I$, we infer
\begin{equation}
|\u_h   - \u_I |_{1,\Omega} \lesssim \sup_{\v_h\in \Vh} \frac{< \f , \v_h>_h - (\f,\v_h)}{|\v_h|_{1,\Omega}}   +
|\u_I - \u_\pi  |_{1,\Omega}   +    |\u_\pi - \u  |_{1,\Omega} .
\label{est4}
\end{equation}
The triangle inequality thus gives
\begin{equation}
|\u   - \u_h |_{1,\Omega} \lesssim \sup_{\v_h\in \Vh} \frac{< \f , \v_h>_h - (\f,\v_h)}{|\v_h|_{1,\Omega}}   +
|\u - \u_I  |_{1,\Omega}   +    |\u - \u_\pi  |_{1,\Omega} .
\label{est5}
\end{equation}
\end{proof}

\begin{remark} Theorem~\ref{error-estimate} applies also
to Problem~\eqref{elast-virt-prob} at the final step $N$. Indeed, it is sufficient to make the choices $\f=\f^N$, $\s_N=\u_h^{N-1}$ in~\eqref{vp2} , and  to identify $\u_h$ in~\eqref{vp2} with $\u_h^N$ in~\eqref{elast-virt-prob}.
\end{remark}

\begin{corollary}
Following the same notation of Theorem \ref{error-estimate}, let moreover $\u \in [H^2(\Omega)]^d$.
Then the linear convergence bound holds
$$
|\u   - \u_h |_{1,\Omega} \lesssim h \: |\u|_{2,\Omega} .
$$
\end{corollary}
\begin{proof}
The results follows immediately combining Theorem \ref{error-estimate} with standard polygonal approximation estimates for the spaces $\VhE$ and $\Ppb_1(E)$, see \cite{volley,Steklov-VEM}.
\end{proof}

\section{Numerical tests}
\label{sec:num}

In the present section we test our virtual method.
In the first two examples (see Sections~\ref{sec:test1} and \ref{sec:test2}), the body occupies the region
$\Omega:=(0,1)^2$, where lengths are expressed in meters. We employ the following types
of mesh (see also Figures~\ref{FIG:1}-\ref{FIG:2}):
\begin{itemize}
\item $\Th^1$: Structured hexagonal meshes.
\item $\Th^2$: Non-structured hexagonal meshes made of convex hexagons.
\item $\Th^3$: Regular subdivisions of the domain in $N\times N$ subsquares.
\item $\Th^4$: Trapezoidal meshes which consist of partitions of the domain
into congruent trapezoids, all similar to the trapezoid with vertexes
$(0,0)$, $(\frac{1}{2},0)$, $(\frac{1}{2},\frac{2}{3})$, and $(0,\frac{1}{3})$.
\end{itemize}

In what follows, $N_h$ denotes the number of vertices in the mesh under consideration.

\begin{figure}
\begin{center}
\begin{minipage}{4.2cm}
\centering\includegraphics[height=4.2cm,
width=4.2cm]{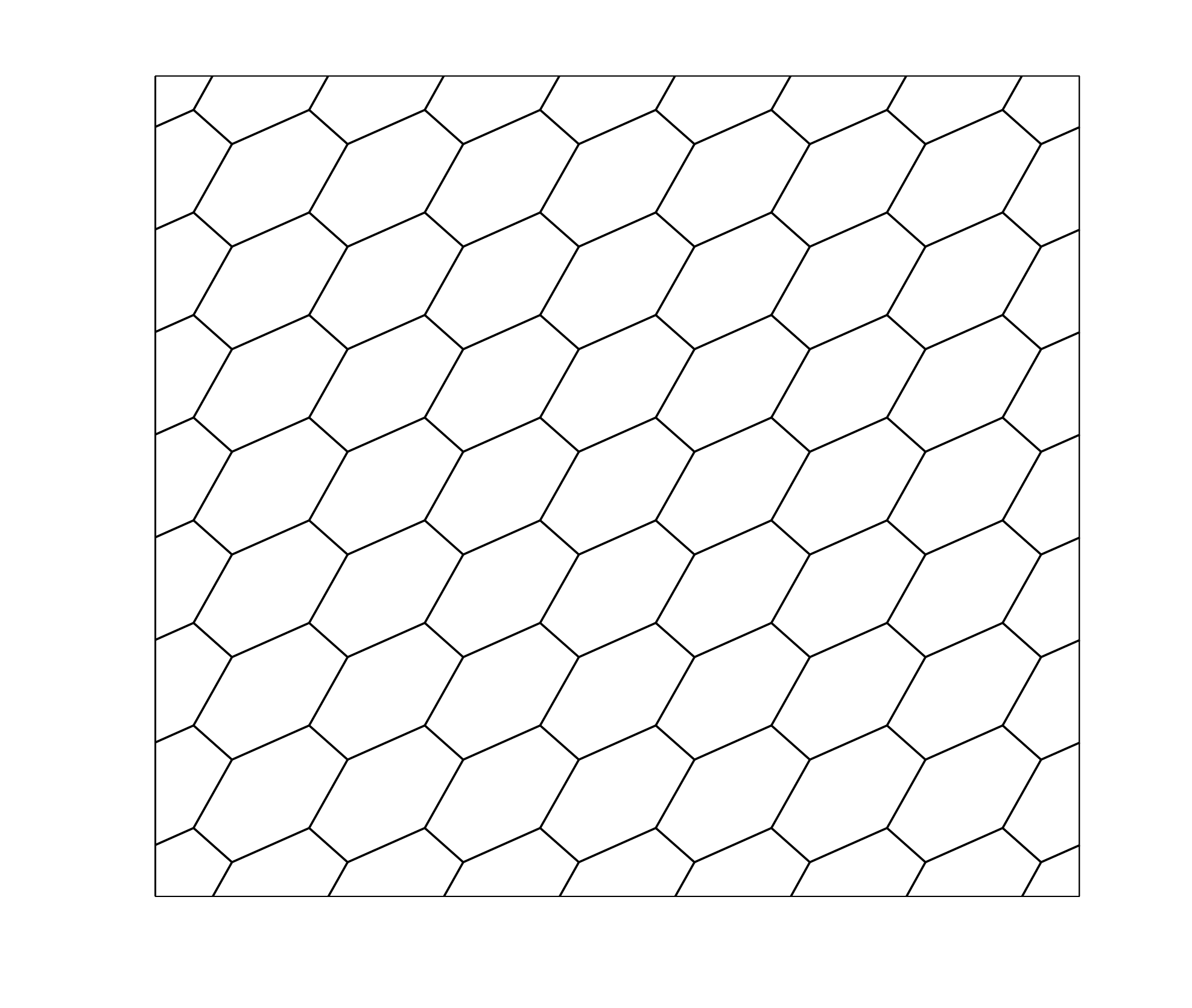}
\end{minipage}
\begin{minipage}{4.2cm}
\centering\includegraphics[height=4.2cm,
width=4.2cm]{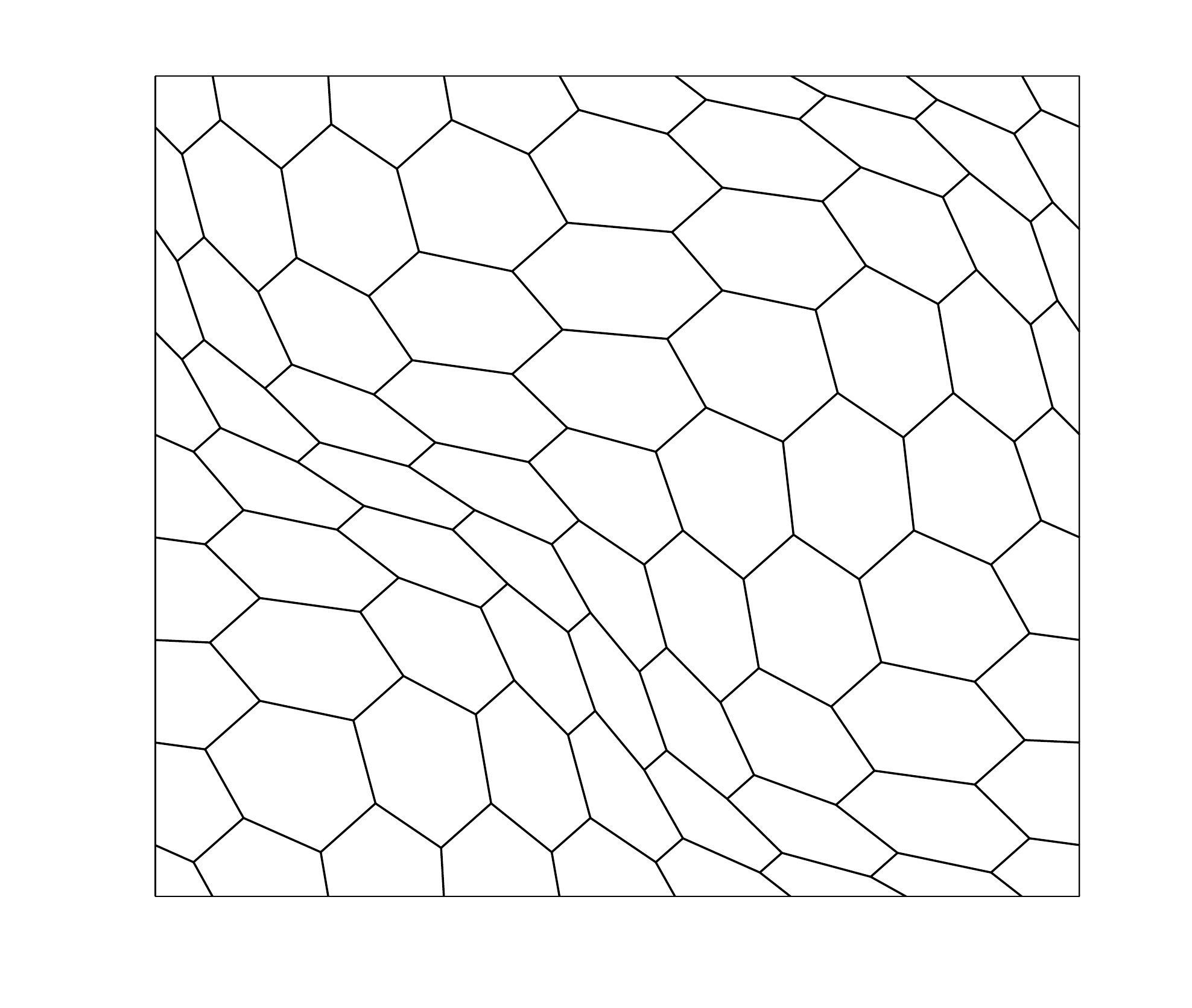}
\end{minipage}
\caption{Sample meshes: $\Th^{1}$ (left) and $\Th^{2}$ (right).}
\label{FIG:1}
\end{center}
\end{figure}
\begin{figure}
\begin{center}
\begin{minipage}{4.2cm}
\centering\includegraphics[height=4.2cm,
width=4.2cm]{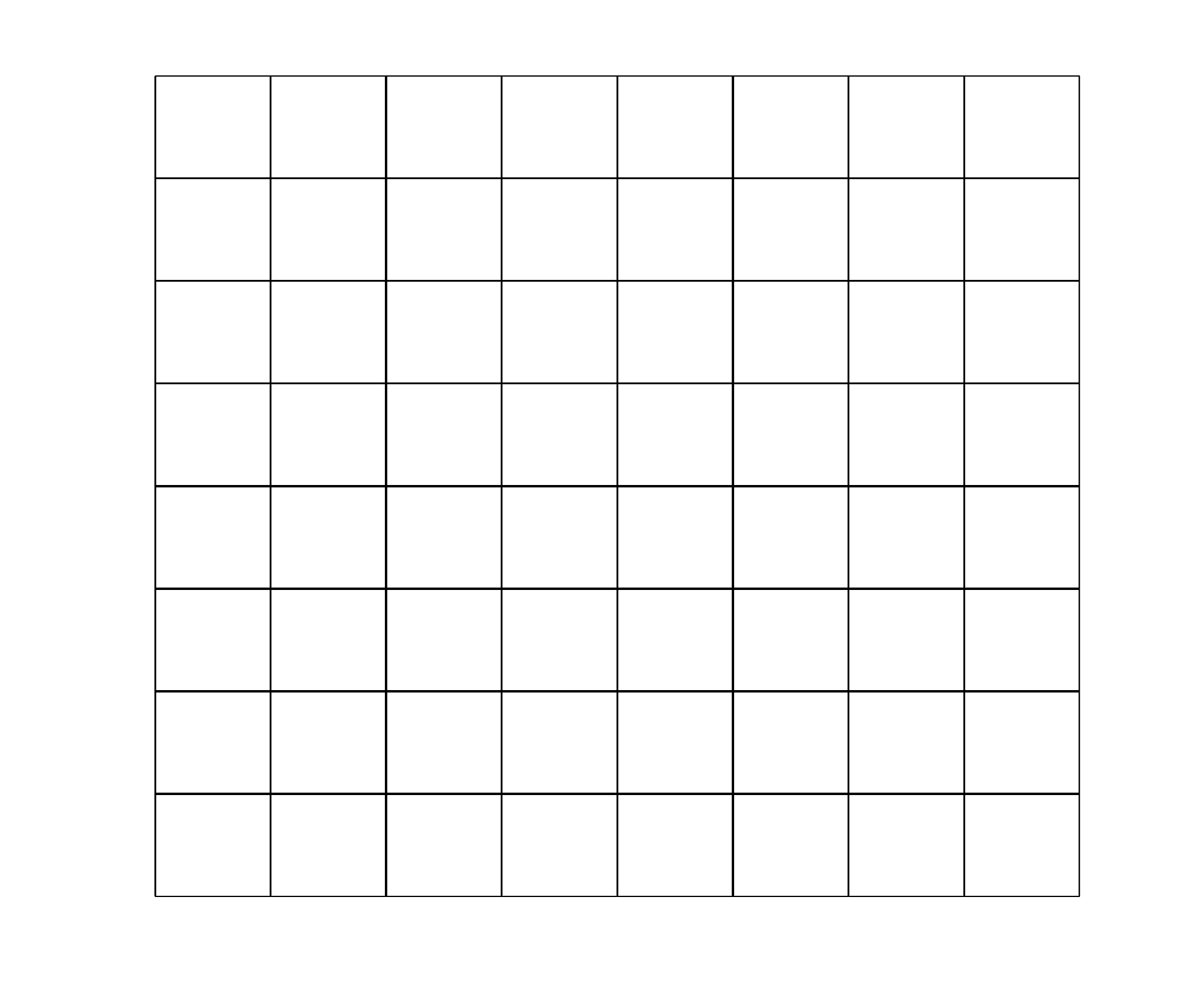}
\end{minipage}
\begin{minipage}{4.2cm}
\centering\includegraphics[height=4.2cm,
width=4.2cm]{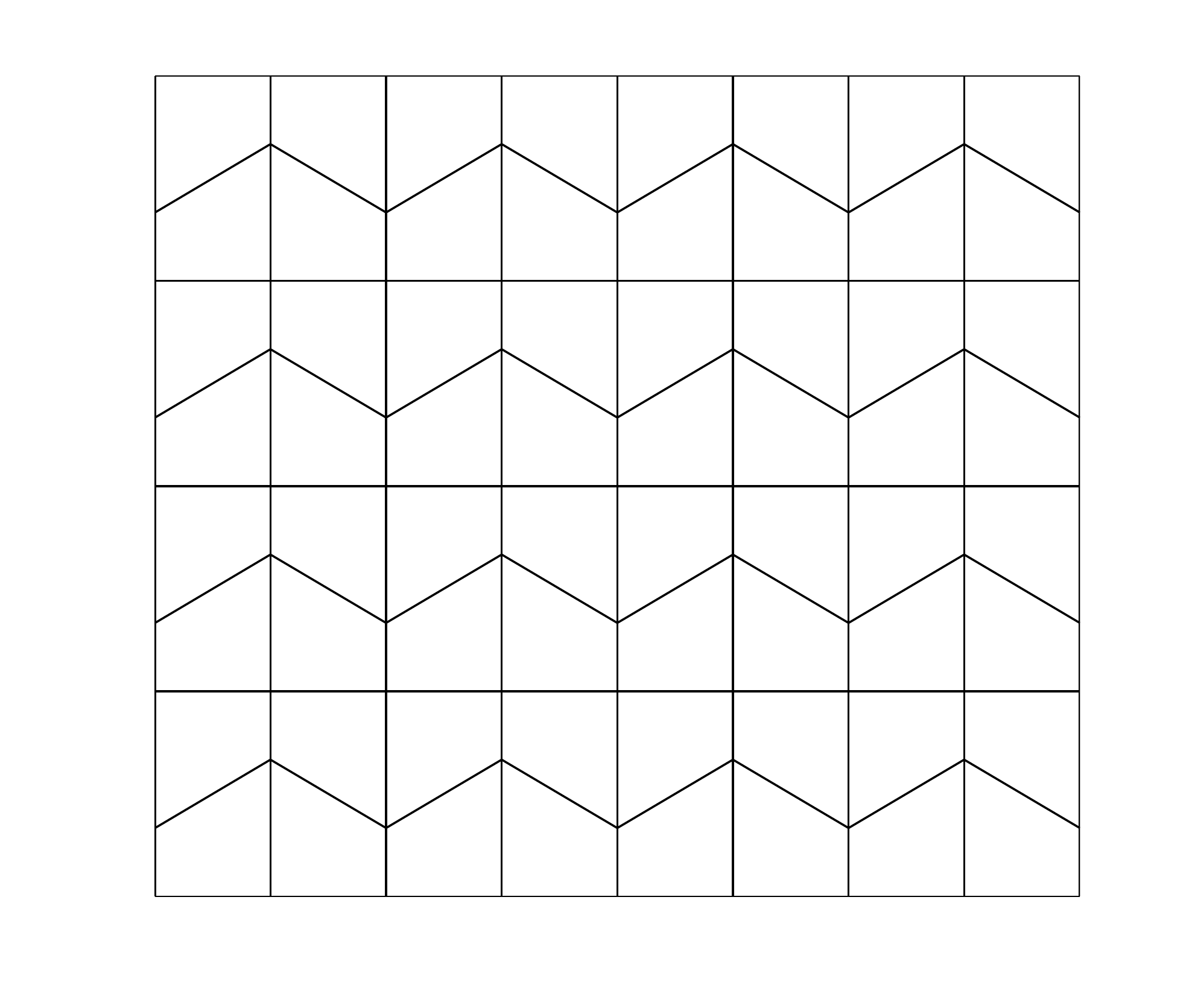}
\end{minipage}
\caption{Sample meshes: $\Th^{3}$ (left) and $\Th^{4}$ (right).}
\label{FIG:2}
\end{center}
\end{figure}

To test the convergence properties of the methods,
we introduce the following discrete maximum norm:
for any sufficiently regular function $\vb$,
\begin{equation}\label{error-norm0}
||| {\vb} |||_{0,\infty}:=\max_{\vv\in {\calV}_h}{| \vb(\vv) |_{\infty}}
\end{equation}
where ${\calV}_h$ represents the set of vertexes of $\Th$
and $| \cdot |_{\infty}$ denotes the $l^{\infty}$ vector norm.
We also introduce the following discrete $H^1$ like norm:
\begin{equation}\label{error-norm1}
||| \vb |||_{1,2}:=\left(\sum_{e\in\Eh}h_e\left\Vert\frac{\partial \vb}{\partial \tb_e}\right\Vert_{0,e}^2\right)^{1/2} ,
\end{equation}
where $\Eh$ and $h_e$ denote the set of edges in the mesh and the length of the edge $e$, respectively. Moreover, $\tb_e$ denotes one of the two tangent vectors to the edge $e$, chosen once and for all.
Accordingly, we denote by
$$E_{0,\infty}^h  := ||| \ub - \ub_h |||_{\infty}\qquad E_{1,2}^h  := ||| \ub - \ub_h |||_{1,2}$$
the corresponding errors and we measure the experimental order of convergence as
\begin{equation*}
R :=-2\frac{\log(E(\cdot)/E'(\cdot))}{\log(N_h/N_{h'})},
\end{equation*}
where $N_h$ and $N_{h'}$ denote the number of vertices in two consecutive
meshes, with corresponding errors $E$ and $E'$.


\subsection{Hencky-von Mises elasticity problem with analytical solution}
\label{sec:test1}

The first constitutive law we consider, taken from \cite{GMR13},
is the non-linear Hencky-von Mises elasticity model,
for which
$$\sigmab = \sigmab(x,\nabla \u(x))=
\tilde{\lambda}(\dev(\boldsymbol{\varepsilon}(\u)))\tr(\boldsymbol{\varepsilon}(\u))I
+2\tilde{\mu}(\dev(\boldsymbol{\varepsilon}(\u)))\boldsymbol{\varepsilon}(\u).$$
Here above, $\tilde{\lambda}$ and $\tilde{\mu}$ are the nonlinear Lam\'e functions,
$\boldsymbol{\varepsilon}(\u):=\frac{1}{2}(\nabla\u+(\nabla\u)^T)$
is the small deformation strain tensor, the symbol
$\tr$ represents the trace operator and $\dev(\boldsymbol{\tau})=||(\boldsymbol{\tau}
-\frac{1}{2}tr(\boldsymbol{\tau})I)||$ is the Frobenius norm of the deviatoric part of the tensor $\boldsymbol{\tau}$.

We take the Lam\'e functions as follows:
$$\tilde{\mu}(\rho):=\frac{3}{4} \left( 1+ (1+\rho^2)^{-1/2}   \right) \cdot 10^4{\rm MPa}
\quad
\text{and}\quad \tilde{\lambda}(\rho):=\frac{3}{4} \left(1- 2 \tilde{\mu}(\rho)\right)\cdot 10^4{\rm MPa}
\quad\forall\rho\in{\mathbb R}^{+},$$
%
This function $\tilde\mu$ corresponds to the
Carreau law for viscoplastic materials. It is easy to verify that the hypotheses at the beginning of Section~\ref{theor-results} are fulfilled by our choice of $\tilde\lambda$ and $\tilde\mu$.
We have taken the load $\f$ such that
the solution $\u$ of Problem \eqref{elast-prob-strong} is given by:
\begin{equation*}
u_{1}(x,y)=u_{2}(x,y)=\sin(\pi x)\sin(\pi y).
\end{equation*}

In Table~\ref{TAB:L1} we report
the convergence history of the virtual method
\eqref{elast-virt-prob} applied to our test problem with
different families of meshes.
The table includes the number of mesh vertices,
the convergence rates $R$, and the discrete errors
$E_{0,\infty}^h$ and $E_{1,2}^h$.


\begin{table}[ht]
\begin{center}
{\scriptsize\begin{tabular}{|c|c|cc|cc|}
\hline
Mesh  &$N_h$  & $E_{0,\infty}^h$ & ${R_{0,\infty}}$ & $E_{1,2}^h$ & ${R_{1,2}}$\\
\hline
             & 64  &3.4192e-2 &-- &4.5675e-1 &--    \\
              & 192  &8.2511e-3 &2.59 &2.4445e-1 &1.14  \\
$\Th^{1}$      & 640  &2.4353e-3 &2.03 &1.2803e-1 &1.07 \\
              & 2304  &6.7066e-4 &2.01 &6.5274e-2 & 1.05 \\
              & 8704  &1.7495e-4 &2.02 &3.2919e-2 & 1.03 \\
              & 33792   &4.4619e-5 &2.01 &1.6527e-2 &1.02  \\
\hline
             & 64  &5.6458e-2  &-- &5.0007e-1 &--    \\
              & 192  &1.9675e-2 &1.92 &2.7166e-1 &1.11  \\
$\Th^{2}$      & 1280 &6.4750e-3 &1.85 &1.4054e-1 &1.09 \\
              & 2304 &2.01403-3 &1.82 &7.1120e-2 &1.06 \\
              & 8704 &5.4860e-4 &1.96 &3.5590e-2 &1.04 \\
              & 33792  &1.4070e-4 &2.01 &1.7817e-2 & 1.02 \\
\hline
               &25  &6.1947e-2 &--   &7.1975e-1 &--   \\
              & 81  &9.3599e-3 &3.21 &3.5627e-1 &1.19  \\
$\Th^{3}$      &578  &1.7576e-3 &2.62 &1.7809e-1 &1.09   \\
              &1089 &4.2329e-4 &2.14 &8.9038e-2 &1.04   \\
             &4225  &1.0516e-4 &2.05 &4.4518e-2 &1.02   \\
             &16641 &2.6254e-5 &2.02 &2.2259e-2 &1.01   \\
\hline
               &25  &1.5401e-1 &-- &1.0516e-0 &--   \\
              & 81  &3.3021e-2 &2.62 &5.3972e-1 & 1.14 \\
$\Th^{4}$      &578  &7.1005e-3 &2.42 &2.7525e-1 & 1.06  \\
              &1089  &1.6650e-3 &2.19 &1.3832e-1 & 1.04  \\
             &4225  &4.1133e-4 &2.06 &6.9382e-2 & 1.02  \\
             &16641  &9.0462e-5 &2.21 &3.2452e-2 & 1.05  \\
             \hline
\hline
\end{tabular}}
\caption{{Approximation of $\u$: convergence
analysis of the virtual method \eqref{elast-virt-prob}}.}
\label{TAB:L1}
\end{center}
\end{table}

We observe from Table~\ref{TAB:L1} that a clear first
order convergence rate in the discrete $H^1$ like norm
and show a quadratic rate in the discrete $L^\infty$ norm.

\subsection{A benchmark elasticity model problem with analytical solution}\label{sec:test2}

In this test case, we select the constitutive load as

$$\sigmab = \sigmab(x,\nabla \u(x))=
\hat{\mu}(\boldsymbol{\varepsilon}(\u))\boldsymbol{\varepsilon}(\u),$$
where $\hat{\mu}$ is defined by the following nonlinear function:
$$\hat{\mu}(\boldsymbol{\varepsilon}(\u)):=3 (1+\Vert\boldsymbol{\varepsilon}(\u)\Vert^2) \cdot 10^4{\rm MPa},$$
with
$$\Vert\boldsymbol{\varepsilon}(\u)\Vert^2=\sum_{i,j=1}^{2}\vert\varepsilon_{ij}\vert^2.$$

We have taken the load $\f$ such that
the solution $\u$ of Problem \eqref{elast-prob-strong} is given by:
\begin{equation*}
u_{1}(x,y)=u_{2}(x,y)=10\sin(\pi x)\sin(\pi y).
\end{equation*}

We remark that this choice does not actually correspond to any elastic material. Instead, it has been chosen as a ``benchmark model'' which does not satisfy the assumption at the begining of Section~\ref{theor-results}: condition~\eqref{sigmacont} does not hold, in particular.

Table~\ref{TAB:L2} shows the convergence history of the virtual method
\eqref{elast-virt-prob} applied to our test problem with
different families of meshes.
The table includes the number mesh vertices,
the convergence rates $R$, and the discrete errors
$E_{0,\infty}^h$ and $E_{1,2}^h$.


\begin{table}[ht]
\begin{center}
{\scriptsize\begin{tabular}{|c|c|cc|cc|}
\hline
Mesh  &$N_h$  & $E_{0,\infty}^h$ & ${R_{0,\infty}}$ & $E_{1,2}^h$ & ${R_{1,2}}$\\
\hline
             & 64  &4.1122e-2  &-- &4.6371e-0 &--    \\
              & 192  &1.7816e-2 &1.52 &2.6318e-0 & 1.03 \\
$\Th^{1}$      & 1280 &5.0006e-3 &2.11 &1.3317e-0 &1.13 \\
              & 2304 &1.2449e-3 &2.17 &6.6288e-1 &1.08 \\
              & 8704 &2.9750e-4 &2.15 &3.3092e-1 &1.04 \\
              &33792  &8.2512e-5 &1.90 &1.6553e-1 &1.02  \\
\hline
            & 64   &8.1685e-2 &-- &5.1698e-0 &--    \\
             &192   &2.3823e-2 &2.24 &2.9790e-0 &1.00  \\
$\Th^{2}$      &1280   &1.4234e-2 &0.86 &1.5553e-0 &1.08 \\
              & 2304   &5.9189e-3 &1.37 &7.6103e-1 &1.12 \\
             & 8704  &1.7906e-3 &1.80 &3.6614e-1 &1.10 \\
              &33792  &4.7067e-4 &1.97 & 1.7981e-1 &1.05  \\
\hline
              & 25  &1.8457e-1 &-- &9.6706e-0 &--    \\
             & 81  &5.2374e-2 &2.14 &4.0009e-0 &1.50  \\
$\Th^{3}$      & 578  &1.5787e-2 &1.89 &1.8538e-0 &1.21 \\
              & 1089  &4.5978e-3 &1.86 &9.0144e-1 &1.09 \\
              & 4225  &1.2340e-3 &1.94 &4.4672e-1 &1.04 \\
              &16641  &3.1086e-4 &2.01 &2.2279e-1 &1.02 \\
\hline
               &25  &1.4957e-1 &-- &11.0527e-0 &--   \\
              & 81  &3.6140e-2 &2.41 &5.4418e-0 &1.20  \\
$\Th^{4}$      &578  &1.1670e-2 &1.78 &2.6376e-0 & 1.13  \\
              &1089 &3.6360e-3 &1.76 &1.3130e-0 & 1.05  \\
             & 4225  &1.1048e-3 &1.76 &6.5565e-1 & 1.02  \\
             &16641 &3.1365e-4 &1.83 &3.2786e-1 &  1.01 \\
             \hline
\hline
\end{tabular}}
\caption{{Approximation of $\u$: convergence
analysis of the virtual method \eqref{elast-virt-prob}}.}
\label{TAB:L2}
\end{center}
\end{table}

Once more, a quadratic order of convergence in the discrete $L^\infty$ norm
and a linear order convergence rate in the discrete $H^1$ like norm
can be clearly appreciated from Table~\ref{TAB:L2}.

We now consider the same $\Omega$ and the same constitutive law, but we choose a couple of different loads. The purpose is now to
show the importance of updating the choice of the stability constant appearing in the elastic form \eqref{a-local}, for instance by employing the recipe detailed in~\eqref{alfa} (see Remark~\ref{rem:alfa}). Therefore, we consider two different external forces, compatible with the following two analytical solutions:

$$
\begin{aligned}
& \textrm{Case 1:} \ \u = \Big( x(1-x)y(1-y) , x(1-x)y(1-y) \Big)^T , \\
& \textrm{Case 2:} \  \u = 80*\Big( x(1-x)y(1-y) , x(1-x)y(1-y) \Big)^T .
\end{aligned}
$$
We notice that in Case 1 the solution gives rise to deformations of moderate magnitude, while in Case 2 much larger deformations occur.
We consider a single family of three regular Voronoi meshes, generated using the algorithm in \cite{polymesher}.
Moreover, we choose the following relative error measure, involving both the displacement components at all the vertices $\mathsf{v}$ of the mesh:
$$
E_{\infty} = \frac{\max_{\mathsf{v} \in \Th, \: i=1,2} |u_i(\mathsf{v}) - (u_{h})_i(\mathsf{v})|}{\max_{\mathsf{v} \in \Th, \: i=1,2}|u_i(\mathsf{v})|} .
$$

In Table \ref{tab:new1} we report the relative errors computed for Case 1, using both the updated scalings introduced in~\eqref{alfa} and a fixed scaling. We notice that convergence is attained for both the strategies of the scaling choice.

In Table \ref{tab:new2} we report the relative errors computed for Case 2, using both the updated scalings introduced in~\eqref{alfa} and a fixed scaling. We notice that for this case, convergence is attained when using the updating strategy, while choosing a fixed scaling provides unsatisfactory results. In particular, on the finest mesh the error is still around $20\%$. Moreover, the solution is highly oscillating due to the presence of unstable numerical modes (figure not shown).

\begin{table}[ht]
\begin{center}
\begin{tabular}{|c|c|c|c|}
\hline
Mesh  & $N_h$ &Updated $\alpha_E$&Fixed $\alpha_E$ \\
\hline
Mesh 1   & 199      & $1.715e{-2}$  & $1.174e{-2}$   \\
Mesh 2   & 800      & $3.580e{-3}$  & $3.392e{-3}$  \\
Mesh 3   & 3179    & $1.287e{-3}$   & $8.946e{-4}$ \\
\hline
\end{tabular}
\caption{Case 1: relative errors for the updated and fixed choice of the scaling.}
\label{tab:new1}
\end{center}
\end{table}

\begin{table}[ht]
\begin{center}
\begin{tabular}{|c|c|c|c|}
\hline
Mesh  & $N_h$  &Updated $\alpha_E$&Fixed $\alpha_E$ \\
\hline
Mesh 1   & 199       & $2.384e{-2}$ & $2.685e{0}$  \\
Mesh 2   & 800       & $9.299e{-3}$ &   $9.555e{-1}$  \\
Mesh 3   & 3179     & $3.132e{-3}$  & $2.090e{-1}$ \\
\hline
\end{tabular}
\caption{Case 2: relative errors for the updated and fixed choice of the scaling.}
\label{tab:new2}
\end{center}
\end{table}

\subsection{Von Mises plasticity}
\label{sec:plast}

In the present section we show a numerical example for an inelastic material, von Mises plasticity with linear hardenings.
We consider the classical problem of a strip with circular hole in plain strain regime under enforced displacements of $\delta$ amplitude at two ends. Due to the symmetry of the problem, we can consider one quarter of the strip, as depicted in figure \ref{fig:strip} (left).
\begin{figure}[ht]
\begin{center}
\includegraphics[width=6cm, height=8 cm]{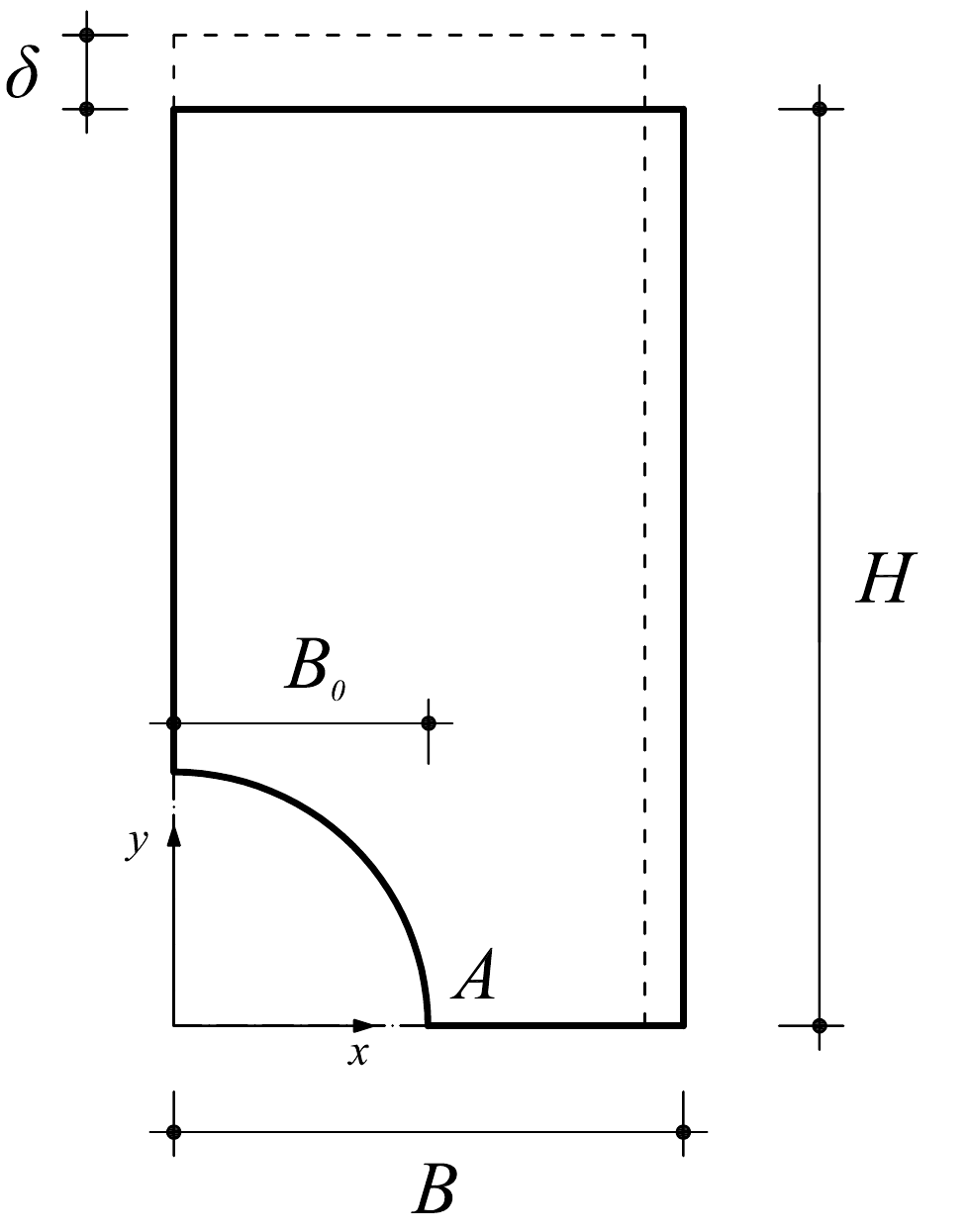}
\includegraphics[width=5.5cm, height=8.7cm]{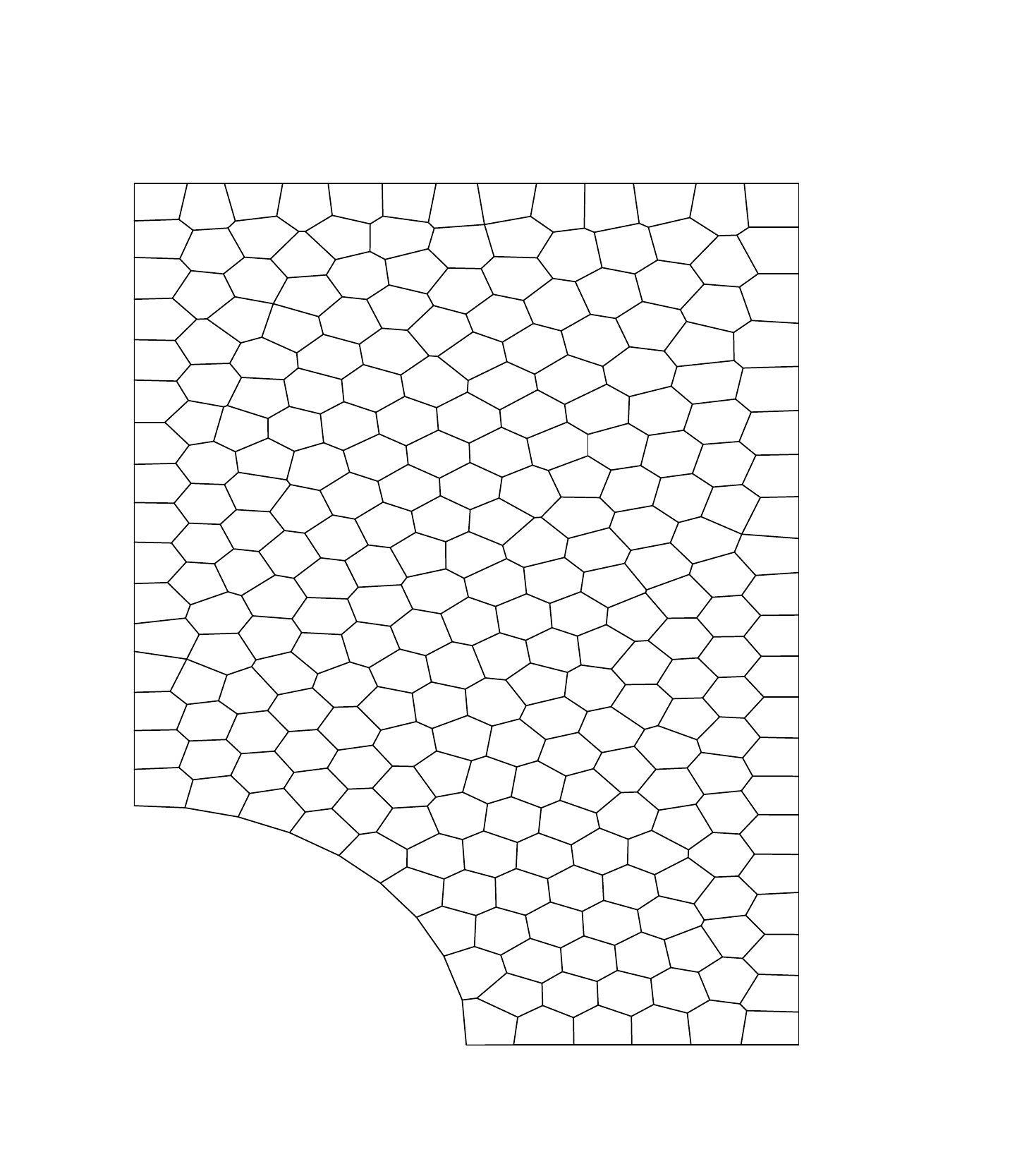}
\caption{Left: depiction of the geometry for the perforated strip problem. Right: sample Voronoi mesh V2.}
\label{fig:strip}
\end{center}
\end{figure}
The geometric data are
$$
B=100 \: \textrm{mm}, \ H= 180 \: \textrm{mm}, \  B_0 = 50 \: \textrm{mm} , \ \delta=10 \: \textrm{mm} .
$$
We consider a $J_2$ plasticity model with linear kinematic and isotropic hardenings (see for instance \cite{SimoHughes}) with material parameters
$$
E=70 \:\textrm{MPa}, \
\nu = 0.2 \:\textrm{MPa}, \
\sigma_{y,0}= 0.8 \:\textrm{MPa}, \
H_{\rm iso} = 10 \:\textrm{MPa}, \
H_{\rm kin} = 10 \:\textrm{MPa} .
$$
For comparison purposes, we take as ``exact solution'' one obtained with linear finite elements on a fine triangular mesh with $45312$ elements. Note that, since the considered model includes hardenings, there is no risk of volumetric locking and thus triangular elements are a good choice.
We solve the problem on a sequence of four Voronoi meshes (mesh V1 to mesh V4) generated with the code PolyMesher \cite{polymesher}. We depict a sample mesh V2 in figure \ref{fig:strip} (right) while the number of vertices in each grid can be found in Table \ref{tab:strip}.
In all cases we use the incremental loading procedure described in Section \ref{sec:2.3} with 100 time-steps. At each time step the constitutive law is solved using a classical radial return map algorithm (see for instance \cite{SimoHughes}, Chapter 3).
For each mesh we show the following values in Table \ref{tab:strip}:
\begin{itemize}
\item The vertical displacement at the point A of coordinates $(0\textrm{mm},50\textrm{mm})$, where the axes origin is at the center of the hole;
\item the horizontal displacement at the point B of coordinates $(50\textrm{mm},0\textrm{mm})$;
\item the maximum stress $\sigma_{\rm max}$;
\item the total stress $\sigma_T$, i.e. the integral over $\Omega$ of the stress amplitude $|| \sigma || = \big( \sum_{i,j=1,2} |\sigma_{ij}|^2\big)^{1/2}$.
\end{itemize}

\begin{table}
\begin{center}
\begin{tabular}{|c|c|c|c|c|c|}
\hline
Mesh  & $N_h$  & Displ. A & Displ. B & $\sigma_{\rm max}$  & $\sigma_T$ \\
\hline
V1             & 129        & 0.7839 & -0.3181 & 3.3842 & 244.2324 \\
V2             & 511      & 0.8173 & -0.3928 & 4.1354 & 240.1062\\
V3             & 2032    & 0.8253 & -0.4212 & 4.4266 & 238.7653 \\
V4             & 8131    & 0.8277 & -0.4300 & 4.7755 & 238.3688 \\
Reference  & 22921  & 0.8284 & -0.4334  & 4.9891 & 238.2631 \\
\hline
\end{tabular}
\caption{Number of mesh vertices, displacements at points A and B, maximum stress and total stress for the four Voronoi meshes and for a reference value obtained with a fine triangular mesh.}
\label{tab:strip}
\end{center}
\end{table}
Note that, on purpose, in Table \ref{tab:strip} we consider quantities for which is easy to obtain convergence (displacement at point A ad total stress) and other ones for which is harder (displacement at point B and maximum stress). In all cases we can appreciate the convergence of the method towards the reference values; finer Voronoi meshes would be needed for a better approximation of the maximum stress.

In figure \ref{fig:gamma} we depict the value of the plastic consistency parameter $\gamma$ for the V4 and for the fine reference mesh. The parameter $\gamma$ indicates if and how much plastification has occurred locally for the material; we refer again to \cite{SimoHughes} for a detailed description of the model. Again, the results for the proposed method are in good accordance with the reference one.

\begin{figure}[t]
\begin{center}
\includegraphics[width=7cm, height=7cm]{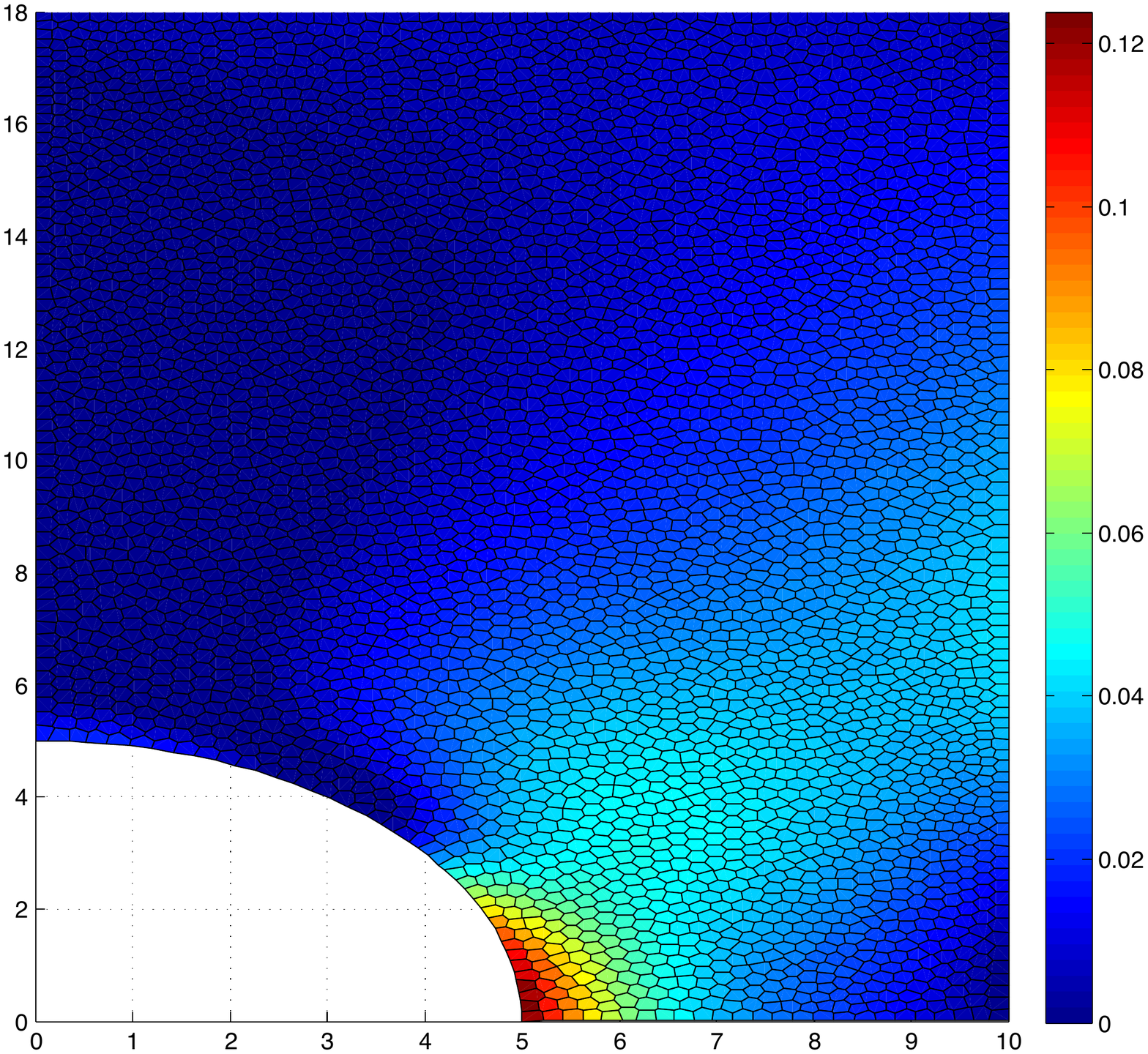}
\includegraphics[width=7cm, height=7cm]{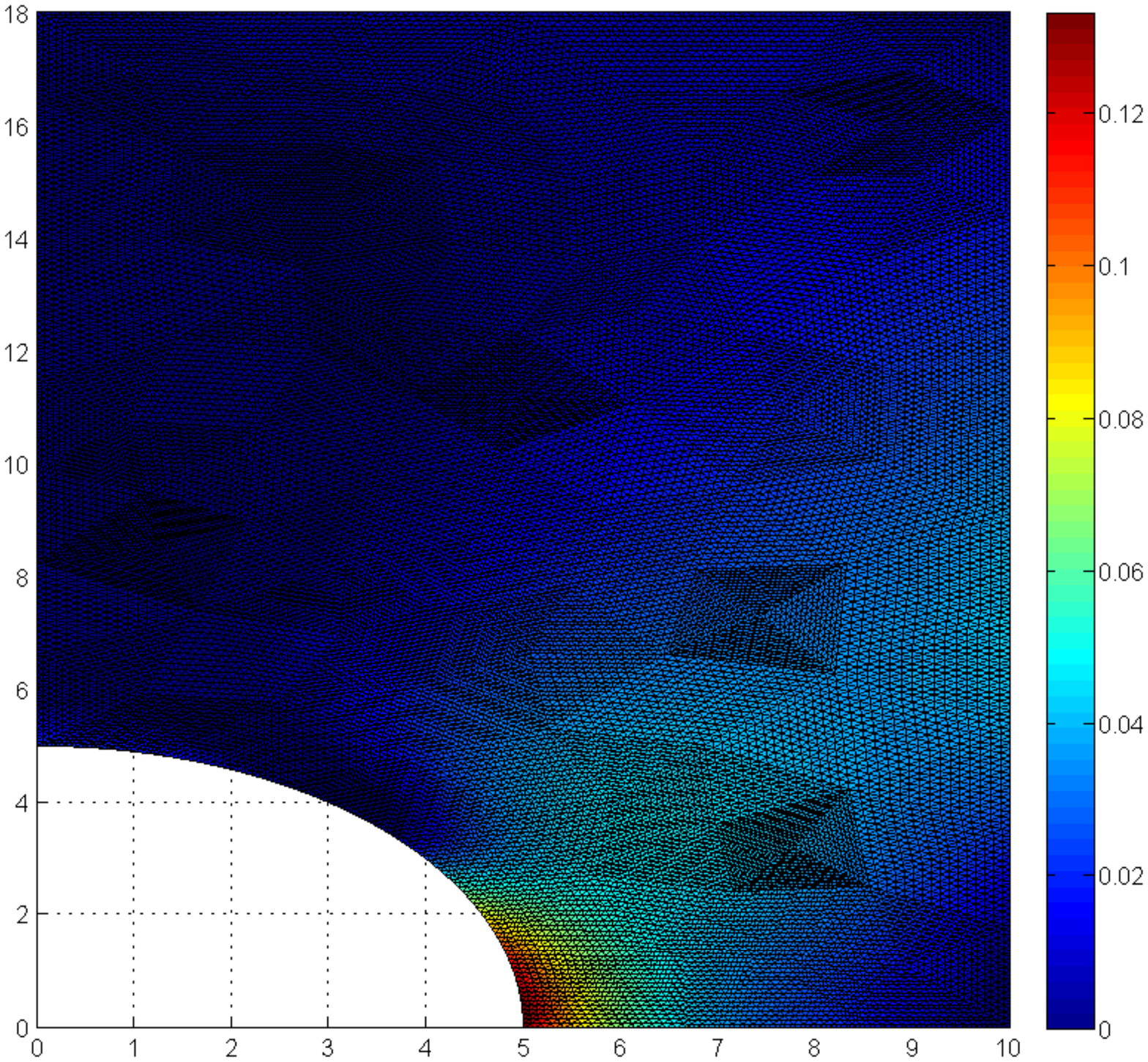}
\caption{Depiction of the plastic flow $\gamma$, mesh V4 on the
left and reference triangular mesh on the right.}
\label{fig:gamma}
\end{center}
\end{figure}

%
%
%
%
%

\subsection{Finite strain elasticity}\label{s:largedef}

The method detailed in Sections~\ref{sec:2.1}-\ref{sec:2.2} can also be applied to elastic problems in a large strain regime. However, we remark that the complexity of the finite elasticity problem requires a much deeper design and analysis than the one here presented. Therefore, the following discussion should be intended only as a very preliminary study towards the VEM discretization of large deformation elastic problems.

We here focus on neo-Hookean hyperelastic materials, but different constitutive laws could be considered. Following a material description (see~\cite{BonetWood, Ciarlet, MarsdenHughes}, for instance),  the variational formulation of the elastic large deformation problem reads as in~\eqref{elast-prob}:

\begin{equation}\label{elast-prob-largdef}
\left\{
\begin{aligned}
& \textrm{Find } \u \in \V \textrm{ such that } \\
& \int_{\Omega} \Pb(x,\nabla \u(x)) : \nabla \v(x) \dx = \int_{\Omega} \f(x)\cdot\v(x) \dx
\quad \forall \v \in \W ,
\end{aligned}
\right.
\end{equation}
where the first Piola-Kirchhoff stress tensor $\Pb(x,\nabla \u(x))$ is not necessarily symmetric. As for Problem~\eqref{elast-prob}, in~\eqref{elast-prob-largdef} the symbol $\V$ denotes the space of admissible displacements and $\W$ the space of its variations.
A homogeneous neo-Hookean material is described by the constitutive law:

\begin{equation}\label{neo-hookean}
\begin{aligned}
\Pb(x,\nabla \u(x))= &\mu [ ( \Id + \nabla\u) + ( \Id + \nabla\u)^{-T} ] \\
& + \lambda \Theta(\det ( \Id + \nabla\u)  )\pi( \det ( \Id + \nabla\u) ) ( \Id + \nabla\u)^{-T} .
\end{aligned}
\end{equation}
Above, $\lambda$ and $\mu$ are given constants, $\Theta: {\mathbb R}^+ \longrightarrow  {\mathbb R}$ is a suitable smooth function, and
$\pi$ is defined as
\begin{equation}\label{pifunction}
 \pi(s) = \Theta'(s) s\ .
 \end{equation}
Here, we choose $\Theta(s)= s - 1$, so that $\pi(s)=1$.

A possible virtual method for Problem~\eqref{elast-prob-largdef} can be designed exactly as in Sections~\ref{sec:2.1}-\ref{sec:2.2} , simply by systematically substituting $\Pb$ in place of $\sigmab$.

We test the method considering a square block of side length $1{\rm m}$, which initially occupies the region $\Omega=(0,1)^2$. We impose clamped boundary conditions on the side $\Gamma_c= \{ 0 \}\times[0,1]$, while the remaining part of the boundary is free. The material parameters are chosen as $\mu = 2.6316\cdot 10^4{\rm MPa}$ and $\lambda = 5.1086\cdot 10^4{\rm MPa}$. The load is given by $\f =(1,0)^T10.5\cdot 10^{10}{\rm N}/{\rm m}^3$.

Table~\ref{tab:large} displays the computed displacements of the material point $P=(1,1)^T$, when using triangular (T1,...,T4), quadrilateral (Q1,...,Q4), and hexagonal Voronoi (V1,...,V4) meshes. A reference solution at the same point, obtained with a very fine triangular mesh of $70344$ elements,  corresponding to $35459$ mesh vertices,  is also reported. Finally, Figure~\ref{fig:large} depicts the deformed body when using the triangular mesh T2, the square mesh Q2 and the hexagonal Voronoi mesh V2 of Table~\ref{tab:large}. We notice that for every considered scheme, convergence to the reference solution occurs, and the deformed shapes appear to be sensible.


\begin{table}[ht]
\begin{center}
\begin{tabular}{|c|c|c|c|}
\hline
Mesh  & $N_h$ & $x-$Displ. at $P$ & $y-$Displ. at $P$ \\
\hline
T1             &   55       & 0.9865 & -0.0438  \\
T2             &   183    & 1.0615 & -0.0398    \\
T3             &    727   & 1.0848 & -0.0358  \\
T4             &    2810  & 1.0967 & -0.0354   \\
\hline
\hline
Q1             &    49     & 0.9979 & -0.0736  \\
Q2             &      196  & 1.0730 & -0.04791 \\
Q3             &  784    & 1.0950 & -0.0391   \\
Q4             &   3025   & 1.1005 & -0.0364  \\
\hline
\hline
V1             &   52     & 0.9125 & -0.0673  \\
V2             &    199  & 1.0344 & -0.0520   \\
V3             &  800   & 1.0722 & -0.0408    \\
V4             &    3179 & 1.0918 & -0.0368  \\
\hline
\hline
Reference  &  35459  &  1.1018    &  -0.0353    \\
\hline
\end{tabular}
\caption{ Computed displacements using triangular (T1,...,T4),
square (Q1,...,Q4), and hexagonal Voronoi (V1,...,V4) meshes.}
\label{tab:large}
\end{center}
\end{table}

\begin{figure}[ht]
\begin{center}
\includegraphics[width=4.5cm, height=4.5cm]{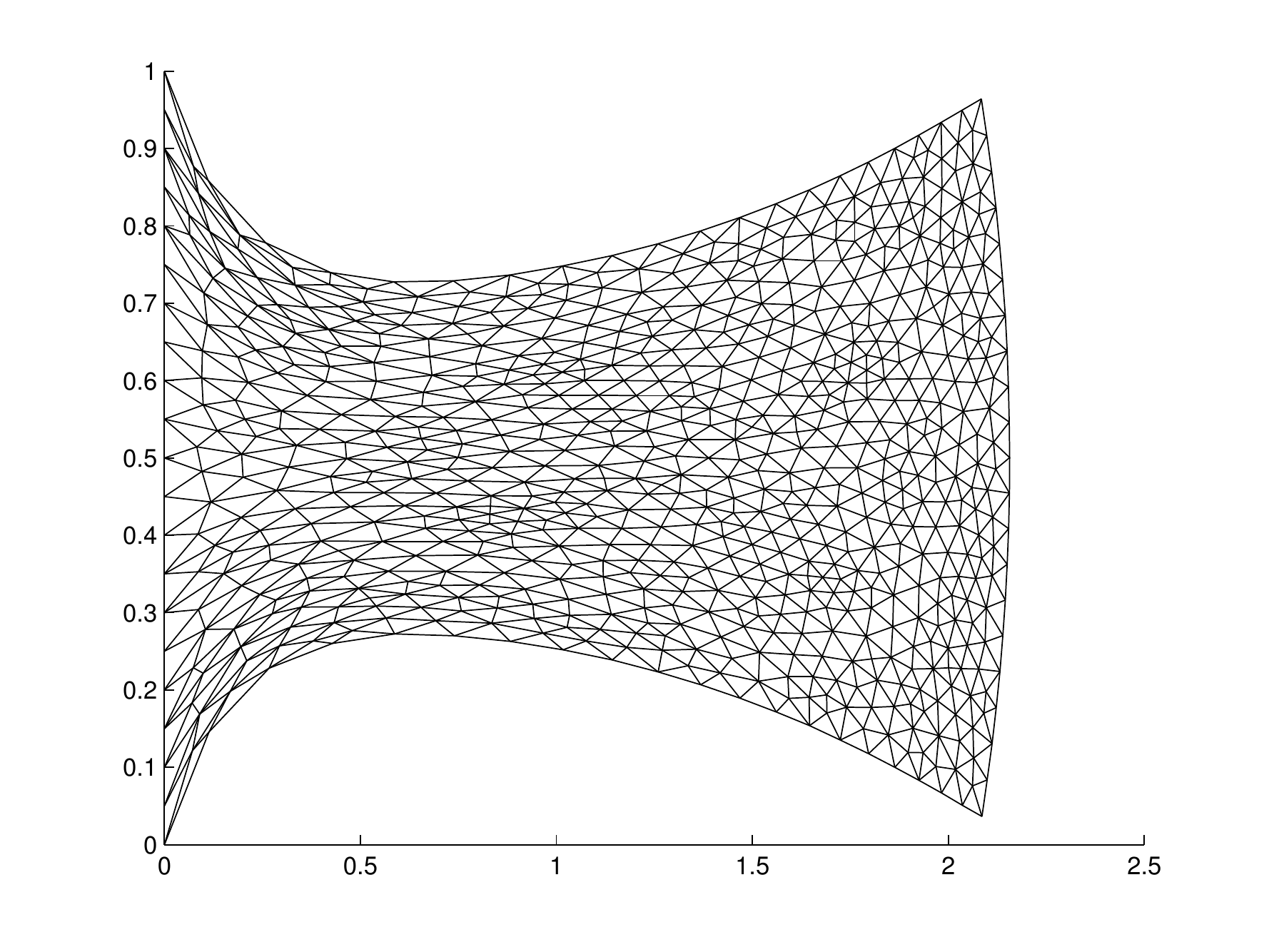}
\includegraphics[width=4.5cm, height=4.5cm]{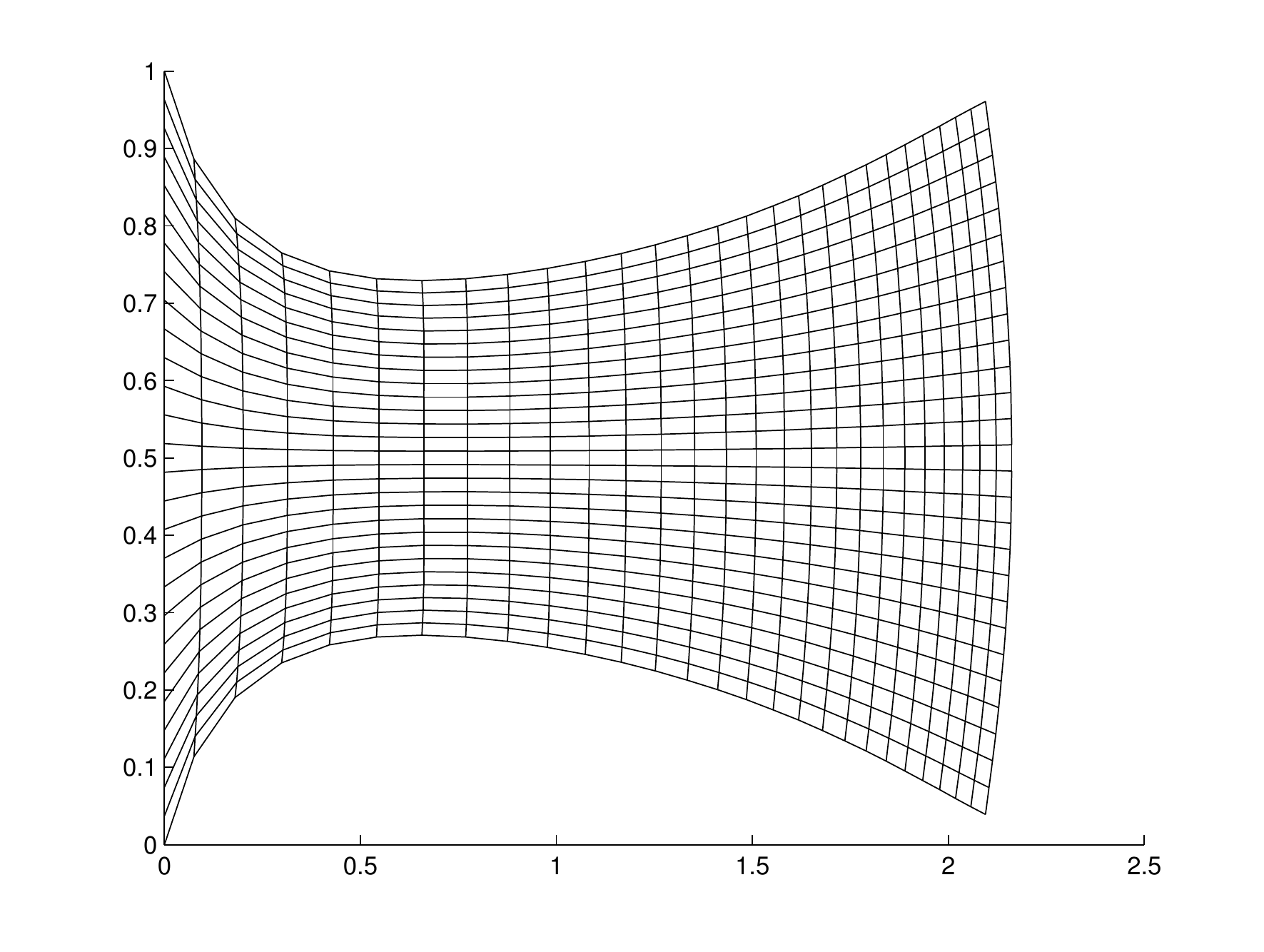}
\includegraphics[width=4.5cm, height=4.5cm]{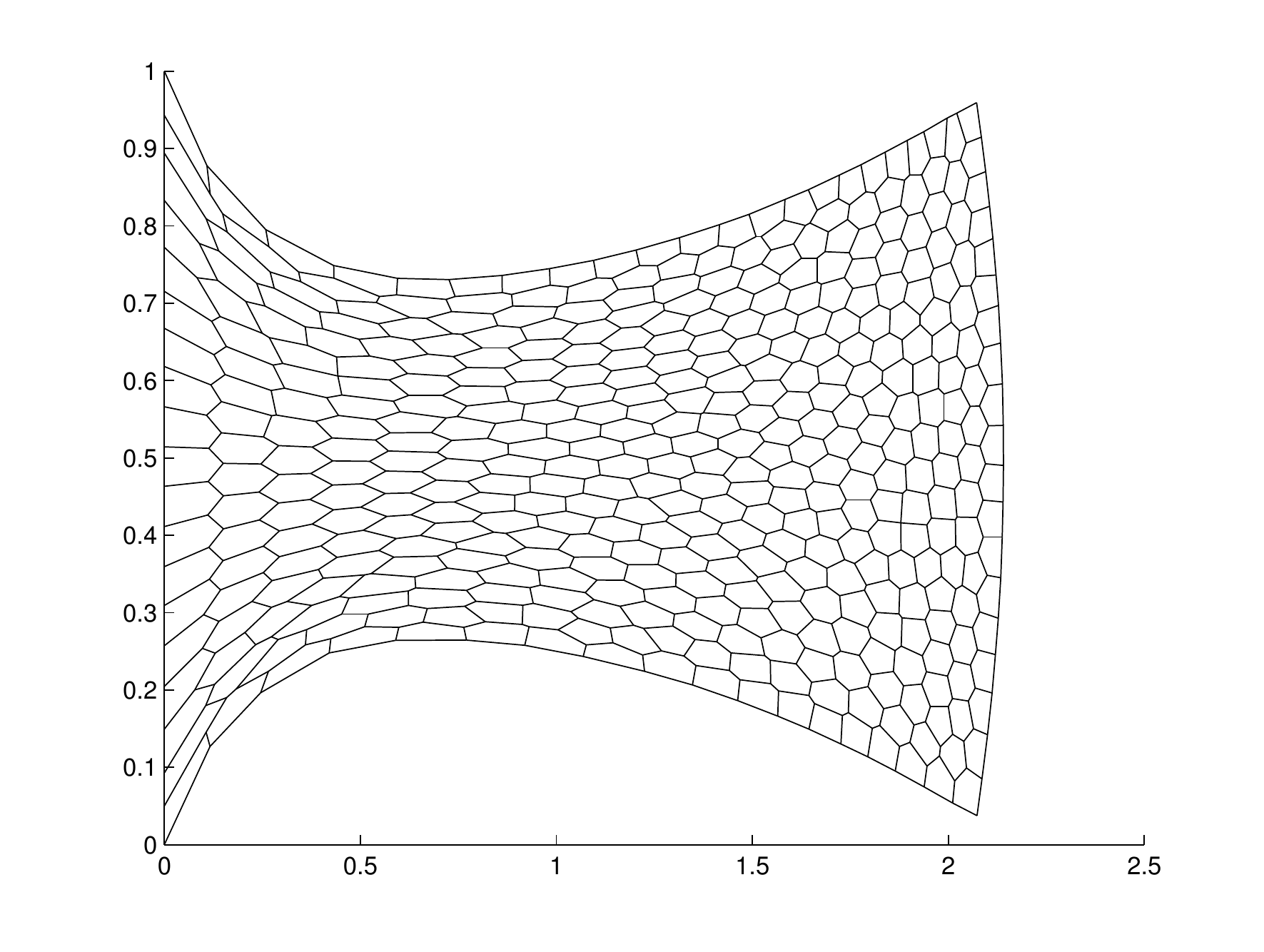}
\caption{Deformed body obtained with the triangular mesh T2 (left), the square mesh Q2 (center), and the hexagonal Voronoi mesh V2 (right).}
\label{fig:large}
\end{center}
\end{figure}


\section{Conclusions}\label{sec:concls}
We have presented a Virtual Element Method to deal with fairly
general non-linear elastic and inelastic problems. Our scheme is
based on a low-order approximation of the displacement field,
together with a suitable treatment of the numerical displacement
gradient. The proposed method allows for general polygonal/polyhedral meshes, is efficient in terms of number of
applications of the constitutive law, and can make use of any
standard black-box constitutive law algorithm. We have presented
several numerical tests assessing the computational performance of
the proposed methodology. However, we remark that this study is
intended as a first step towards the design of efficient Virtual
Element Methods for non-linear Computational Mechanics problems.
Many possible extensions and improvements could be of interest.
For instance, large deformation problems require a much deeper
investigations, and other inelastic cases such as perfect
plasticity or damage could be considered.

\medskip

\noindent{\bf Acknowledgements.}
D. Mora was partially supported by CONICYT-Chile through FONDECYT
project No. 1140791 and by project Anillo ACT 1118 (ANANUM).

\bibliographystyle{plain}
\bibliography{VEM-biblio.bbl}

\begin{thebibliography}{10}

\bibitem{enhanced}
B.~Ahmed, A.~Alsaedi, F.~Brezzi, L.D. Marini, and A.~Russo.
\newblock Equivalent {P}rojectors for {V}irtual {E}lement {M}ethods.
\newblock {\em Comput. Math. Appl.}, 66(3):376--391, 2013.

\bibitem{volley}
L.~Beir\~ao~da Veiga, F.~Brezzi, A.~Cangiani, G.~Manzini, L.~D. Marini, and
  A.~Russo.
\newblock Basic principles of {V}irtual {E}lement {M}ethods.
\newblock {\em Math. Models Methods Appl. Sci.}, 23:119--214, 2013.

\bibitem{elasticity}
L.~Beir\~ao~da Veiga, F.~Brezzi, and L.~D. Marini.
\newblock {V}irtual {E}lements for linear elasticity problems.
\newblock {\em SIAM J. Numer. Anal.}, 51:794--812, 2013.

\bibitem{hitchhikers}
L.~Beir\~ao~da Veiga, F.~Brezzi, L.~D. Marini, and A.~Russo.
\newblock The {H}itchhikers {G}uide to the {V}irtual {E}lement {M}ethod.
\newblock {\em Math. Models Methods Appl. Sci.}, 24(8):1541--1573, 2014.

\bibitem{MFD-book}
L.~Beir\~{a}o~da Veiga, K.~Lipnikov, and G.~Manzini.
\newblock {\em The Mimetic Finite Difference Method for Elliptic Problems}.
\newblock Springer, series MS\&A (vol. 11), 2014.

\bibitem{arbitrary}
L.~Beir\~ao~da Veiga and G.~Manzini.
\newblock A {V}irtual {E}lement {M}ethod with arbitrary regularity.
\newblock {\em IMA J. Numer. Anal.}, 34(2):759--781, 2014.

\bibitem{Wriggers-2014}
S.~Biabanaki, A.~Khoei, and P.~Wriggers.
\newblock Polygonal finite element methods for contact-impact problems on
  non-conformal meshes.
\newblock {\em Comp. Meth. Appl. Mech. Engrng.}, 269:198--221, 2014.

\bibitem{Boffi-Brezzi-Fortin}
D.~Boffi, F.~Brezzi, and M.~Fortin.
\newblock {\em Mixed Finite Element Methods and Applications}.
\newblock Springer-Verlag, Berlin Heidelberg, 2013.

\bibitem{BonetWood}
J.~Bonet and R.D. Wood.
\newblock {\em Nonlinear {C}ontinuum {M}echanics for {F}inite {E}lement
  {A}nalysis}.
\newblock Cambridge University Press; {2} edition, 2008.

\bibitem{Brezzi-Lipnikov-Shashkov:2005}
F.~Brezzi, K.~Lipnikov, and M.~Shashkov.
\newblock Convergence of the mimetic finite difference method for diffusion
  problems on polyhedral meshes.
\newblock {\em SIAM J. Numer. Anal.}, 43(5):1872--1896, 2005.

\bibitem{kirchhoff}
F.~Brezzi and L.D. Marini.
\newblock Virtual {E}lement {M}ethod for plate bending problems.
\newblock {\em Comput. Methods Appl. Mech. Engrg.}, 253:455--462, 2012.

\bibitem{TPPM10}
Talischi C., Paulino G.H., Pereira A., and Menezes I.F.M.
\newblock Polygonal finite elements for topology optimization: A unifying
  paradigm.
\newblock {\em {I}nt. {J}. for {N}um. {M}eth. {E}ngnr.}, 82:671--698, 2010.

\bibitem{Paulino-nonlinear-polygonal}
H.~Chi, C.~Talischi, O.~Lopez-Pamies, and G.H. Paulino.
\newblock Polygonal finite elements for finite elasticity.
\newblock {\em Int. J. Numer. Meth. Engrg.}, 101(4):305--328, 2015.

\bibitem{Ciarlet:78}
P.~G. Ciarlet.
\newblock {\em The Finite Element Method for Elliptic Problems}.
\newblock North-Holland, Amsterdam, 1978.

\bibitem{Ciarlet}
P.G. Ciarlet.
\newblock {\em Mathematical {E}lasticity: {T}hree-dimensional elasticity,
  Volume 1}.
\newblock Elsevier, 1993.

\bibitem{Cockburn-Jay-Lazarov}
B.~Cockburn, J.~Gopalakrishnan, and R.~Lazarov.
\newblock Unified hybridization of discontinuous {G}alerkin, mixed, and
  continuous {G}alerkin methods for second order elliptic problems.
\newblock {\em SIAM J. Numer. Anal.}, 47(2):1319--1365, 2009.

\bibitem{DiPietro-Ern-1}
D.~Di~Pietro and A.~Alexandre~Ern.
\newblock A hybrid high-order locking-free method for linear elasticity on
  general meshes.
\newblock {\em Comput. Methods Appl. Mech. Engrg.}, 283(0):1--21, 2015.

\bibitem{DiPietro-Ern-3}
D.~Di~Pietro and A.~Ern.
\newblock Hybrid high-order methods for variable-diffusion problems on general
  meshes.
\newblock In press, 2014.

\bibitem{Gillette-2}
M.~Floater, A.~Gillette, and N.~Sukumar.
\newblock Gradient bounds for {W}achspress coordinates on polytopes.
\newblock {\em SIAM J. Numer. Anal.}, 52(1):515--532, 2014.

\bibitem{Topology-VEM}
A.L. Gain, G.H. Paulino, L.~Duarte, and I.F.M. Menezes.
\newblock {T}opology {O}ptimization {U}sing {P}olytopes.
\newblock {P}reprint arXiv:1312.7016. {S}ubmitted for publication.

\bibitem{Paulino-VEM}
A.L. Gain, C.~Talischi, and G.H. Paulino.
\newblock On the {V}irtual {E}lement {M}ethod for {T}hree-{D}imensional
  {E}lasticity {P}roblems on {A}rbitrary {P}olyhedral {M}eshes.
\newblock {\em Comp. Meth. Appl. Mech. Engrg.}, 282:132--160, 2014.

\bibitem{GMR13}
G.N. Gatica, A.~M\'arquez, and W.~Rudolph.
\newblock A priori and a posteriori error analyses of augmented twofold saddle
  point formulations for nonlinear elasticity problems.
\newblock {\em Comp. Meth. Appl. Mech. Engrng.}, 264,:23--48, 2013.

\bibitem{HanReddy}
W.~Han and B.D. Reddy.
\newblock {\em Plasticity. {M}athematical {T}heory and {N}umerical {A}nalysis}.
\newblock Springer-Verlag New York, 2013.

\bibitem{Paulino-cracks}
S.E. Leon, D.~Spring, and G.H. Paulino.
\newblock Reduction in mesh bias for dynamic fracture using adaptive splitting
  of polygonal finite elements.
\newblock {\em Int. J. Numer. Meth. Engrng.}, 100:555--576, 2014.

\bibitem{MarsdenHughes}
J.E. Marsden and T.J.R. Hughes.
\newblock {\em Mathematical {F}oundations of {E}lasticity}.
\newblock Dover, 1994.

\bibitem{Steklov-VEM}
D.~Mora, G.~Rivera, and R.~Rodr{\'\i}guez.
\newblock A virtual element method for the {S}teklov eigenvalue problem.
\newblock CI2MA Pre-Publicaci{\'o}n 2014-27, in press on {M}ath. {M}od. {M}eth.
  {A}ppl. {M}ath., 2015.

\bibitem{SimoHughes}
J.~C. Simo and T.~J.~R. Hughes.
\newblock {\em Computational {I}nelasticity}.
\newblock Springer-Verlag New York, 1998.

\bibitem{Sukumar-Tabarraei:2004}
N.~Sukumar and A.~Tabarraei.
\newblock Conforming polygonal finite elements.
\newblock {\em Int. J. Numer. Meth. Engrg.}, 61:2045--2066, 2004.

\bibitem{polymesher}
C.~Talischi, G.H. Paulino, A.~Pereira, and I.F.M. Menezes.
\newblock Poly{M}esher: a general-purpose mesh generator for polygonal elements
  written in {M}atlab.
\newblock {\em Structural and Multidisciplinary Optimization}, 45(3):309--328,
  2012.

\bibitem{Wang-1}
J.~Wang and X.~Ye.
\newblock A weak {G}alerkin finite element method for second-order elliptic
  problems.
\newblock {\em J. Comput. Appl. Math.}, 241:103--115, 2013.

\bibitem{Wang-2}
J.~Wang and X.~Ye.
\newblock A weak {G}alerkin mixed finite element method for second order
  elliptic problems.
\newblock {\em Math. Comp.}, 83(289):2101--2126, 2014.

\end{thebibliography}

\end{document}